\documentclass[12pt, oneside]{amsart}       

\usepackage[T1]{fontenc}
\usepackage{palatino}
  
\usepackage{hyperref}
\usepackage{xcolor}
\hypersetup{
    colorlinks,
    linkcolor={red!50!black},
    citecolor={blue!50!black},
    urlcolor={blue!80!black}
}

\usepackage{graphicx}    
\usepackage{tikz-cd}            
\usepackage{amsthm,amssymb}   
\usepackage{amsaddr} 
\newtheorem{theorem}{Theorem}[section]

\newtheorem{corollary}[theorem]{Corollary}
\newtheorem{definition}[theorem]{Definition}
\newtheorem{lemma}[theorem]{Lemma}
\theoremstyle{remark}
\newtheorem*{remark}{Remark}
\usepackage{tcolorbox}
\newtheorem*{example}{Example}

\renewcommand{\geq}{\geqslant}
\renewcommand{\le}{\leqslant}

\title{Elliptic Loop Spaces}
\author{Emile Bouaziz \quad Adeel A. Khan}
\address{Academia Sinica, Taipei}

\newcommand{\HE}{\mathcal{H}_{E}}

\newcommand{\HA}{\mathcal{H}_{A}}

\newcommand{\DD}{\operatorname{\mathcal{D}}}

\newcommand{\FUN}{\operatorname{Fun}^{\otimes}}

\newcommand{\Fun}{\operatorname{\mathbf{Fun}}^{\otimes}}

\newcommand{\QC}{\operatorname{QCoh}}
\newcommand{\OO}{\mathcal{O}}
\newcommand{\hh}{L_{E}}
\newcommand{\YY}{\mathcal{Y}}

\newcommand{\XX}{\mathcal{X}}

\newcommand{\spc}{\operatorname{spec}}
\newcommand{\Mod}{\operatorname{Mod}}
\newcommand{\GM}{\mathbf{G}_{m}}

\newcommand{\PP}{\mathbf{P}^{1}}
\newcommand{\SSS}{\mathcal{S}_{E}}
\newcommand{\SSSA}{\mathcal{S}_{A}}
\newcommand{\dAlg}{\operatorname{CAlg}_{k}}

\newcommand{\CCC}{\mathbf{\otimes}\textbf{-}\mathbf{Cat}} 

\newcommand{\Ind}{\operatorname{Ind}}
\newcommand{\Perf}{\operatorname{Perf}}
\newcommand{\Map}{\operatorname{Maps}}
\newcommand{\FM}{\Phi}
\newcommand{\mm}{\mathbf{M}_{E}}

\newcommand{\eval}{\operatorname{ev}}
\newcommand{\TT}{\mathbf{T}}

\newcommand{\etr}{\operatorname{tr}^{E}}
\newcommand{\homs}{\operatorname{Hom}}
\newcommand{\hhoms}{\underline{\operatorname{Hom}}}
\newcommand{\hhg}{L^{G}_{E}}

\newcommand{\UU}{\mathcal{U}}
\newcommand{\cl}{{\operatorname{cl}}}
\newcommand{\id}{\mathrm{id}}

\begin{document} 

\begin{abstract} We introduce an elliptic avatar of loop spaces in derived algebraic geometry, completing the familiar trichotomoy of rational, trigonometric and elliptic objects.  Heuristically, the elliptic loop space of $\YY$ is the stack of maps to $\YY$ from a certain exotic avatar $\SSS$ of the elliptic curve $E$, such that the category of quasi-coherent sheaves on $\SSS$ is the convolution category of zero-dimensionally supported coherent sheaves on $E$. For quotient stacks, the structure sheaf of the elliptic loop space gives rise to a theory of equivariant elliptic Hodge cohomology. \end{abstract}

\maketitle

\tableofcontents

\section{Introduction}

This is the first in a series of papers aiming to develop foundations for equivariant elliptic cohomology. In this instalment we will introduce the \emph{elliptic loop space} $\hh(\YY)$ of a stack $\YY$, whose structure sheaf is a primordial incarnation of the elliptic cohomology of $\YY$.
We work over a field $k$ of characteristic zero throughout the introduction.\footnote{This is for expository reasons; see also the discussion in \ref{ssec:other} below.}

\subsection{Rational and trigonometric loop spaces}

There is a standard trichotomy in geometry and representation theory between \emph{rational}, \emph{trigonometric} and \emph{elliptic} objects, corresponding to the natural trichotomy of $1$-dimensional abelian group schemes
\renewcommand{\labelenumi}{(\roman{enumi})}
\begin{enumerate}
    \item \emph{Rational:} $\mathbf{G}_{a}$,
    \item \emph{Trigonometric:} $\mathbf{G}_{m}$,
    \item \emph{Elliptic:} elliptic curves $E$.
\end{enumerate}

We define the \emph{rational} and \emph{trigonometric loop spaces} of a scheme $X$ as the mapping stacks\footnote{Note that these mapping stacks must be formed in the sense of derived algebraic geometry.  In the introduction, schemes are assumed of finite type over $k$.} \[ L^{\mathrm{rat}}(X) := \Map(B\widehat{\mathbf{G}}_{a}, X), \quad L^{\mathrm{trig}}(X) := \Map(B\mathbf{Z}, X). \]

In view of the identification $B\mathbf{Z} \simeq S^1$, $L^{\mathrm{trig}}(X)$ is nothing else but the familiar \emph{free loop space} \[ L(X) = \Map(S^1, X) \simeq X \times_{X\times X} X, \] whose space of functions is \[ \OO(L^{\mathrm{trig}}(X)) \simeq \OO_{X} \otimes^L_{\OO_X\otimes\OO_X} \OO_{X}, \] i.e. the complex of Hochschild chains on $X$. Moreover, taking into account the $S^1$-action on $L(X)$ by loop rotation, we can also recover the cyclic and periodic cyclic homology of $X$.

On the other hand, one can show that $L^{\mathrm{rat}}(X)$ is identified with the \emph{$(-1)$-shifted tangent bundle} $\TT[-1]X$, with space of functions \[ \OO(L^{\mathrm{rat}}(X)) \simeq \operatorname{Sym}^*_{X}(\mathbf{L}_{X}[1]) \simeq \Lambda^*_{X}(\mathbf{L}_{X}), \] where $\mathbf{L}_X$ is the cotangent complex (cf. \cite{Hennion}). Note that this is the derived de Rham complex, without the differential; in other words, it is the complex computing (derived) \emph{Hodge cohomology}. The de Rham differential is encoded by the translation action of $B\widehat{\mathbf{G}}_a$ on $L^{\mathrm{rat}}(X)$. 

The Hochschild--Konstant--Rosenberg theorem can now be interpreted geometrically as a canonical isomorphism \[ L^{\mathrm{rat}}(X) \simeq L^{\mathrm{trig}}(X). \] The respective $S^1$- and $B\widehat{\mathbf{G}}_a$-actions factor through their affinizations $$\operatorname{Aff}(S^1) \simeq B\mathbf{G}_a \simeq \operatorname{Aff}(B\widehat{\mathbf{G}}_a)$$ and the above upgrades to an $B{\mathbf{G}}_a$-equivariant isomorphism. See \cite{TVHKR} or \cite{BNloop}.

Note that the above constructions also make sense for stacks $\YY$. For $\YY$ the classifying stack $BG = */G$ of an algebraic group $G$, for instance, we get the adjoint quotients \[ L^{\mathrm{rat}}(BG) = \mathfrak{g}/G, \quad L^{\mathrm{trig}}(BG) = G/G. \] In fact the isomorphism \[ L^{\mathrm{rat}}(\YY) \simeq \TT[-1]\YY \] still holds for general Artin stacks $\YY$. \footnote{A pointed morphism $(B\widehat{G},*)\rightarrow(\YY,y)$ equivalent to a morphism of Lie algebras $\operatorname{Lie}(G)\rightarrow\TT_{y,\YY}[-1],$ where $\TT_{y,\YY}[-1]$ has the Lie algebra structure constructed in \cite{Hennion}. Taking $G=\mathbf{G}_{a}$ we have the claim. }

The HKR isomorphism, however, does not extend to stacks. Indeed, $\TT[-1]\YY$ is local on $\YY$ with respect to smooth or even fppf covers, whereas as $L(\YY)$ is only local in the Zariski topology. We can understand this difference in cohomological terms. While de Rham cohomology satisfies étale descent, hence is completely determined by its values on schemes, Hochschild homology is by contrast a ``genuine'' invariant of stacks. For global quotient stacks $\YY = X/G$, this is reflected in the difference between Borel-type and genuine equivariance (in the sense of genuine equivariant homotopy theory).

\subsection{\texorpdfstring{$1$}{1}-categorical Cartier duality}

In order to motivate our construction of elliptic loop spaces, we first explain how the rational and trigonometric loop spaces can be regarded as instances of the same general construction. This will also explain the connection with the group schemes $\mathbf{G}_a$ and $\mathbf{G}_m$.

Given an algebraic group $A$, we define the $1$-categorical Cartier dual of $A$ as the group-theoretic mapping stack $$\check{A}:=\operatorname{GroupMap}(A,B\GM).$$ We think of this as a $1$-categorical analogue of the character group $\operatorname{GroupMap}(A, \GM)$.

On the category $\operatorname{Coh}_0(A)$ of coherent sheaves on $A$ with $0$-dimensional support, one can define a convolution product $\star$. There are symmetric monoidal equivalences of categories \begin{align*}
  \big(\operatorname{Coh}(B\widehat{\mathbf{G}}_{a}),\otimes\big) \simeq~&\big(\operatorname{Coh}_{0}(\mathbf{G}_{a}),\star\,\big),\\
  \big(\operatorname{Coh}(B\mathbf{Z}),\otimes\big) \simeq~&\big(\operatorname{Coh}_{0}(\GM),\star\,\big),
\end{align*}
These equivalences encode the $1$-categorical Cartier dualities between $\mathbf{G}_a$ and $B\widehat{\mathbf{G}}_{a}$, and $\GM$ and $B\mathbf{Z}$, respectively.

\subsection{Elliptic loop spaces}

As we have seen, the rational and trigonometric loop spaces $L^{\mathrm{rat}}(\YY)$ and $L^{\mathrm{trig}}(\YY)$ may be regarded as geometrizations of de Rham cohomology and periodic cyclic homology, respectively. The goal of this paper is to introduce an elliptic analogue.

\subsubsection{Formal syntax}

Given an elliptic curve $E$ over $k$, the main output of our construction will be a canonical functor \[ \YY \mapsto L_E(\YY), \] valued in derived\footnote{As we have already seen in the rational and trigonometric cases, the use of derived algebraic geometry is inevitable (even for $\YY$ underived).} stacks equipped with an $E$-action.

For sufficiently nice stacks (such as global quotients), we will see that $\hh(\YY)$ is a Zariski cosheaf in $\YY$, and preserves fibre squares. We regard these as geometrizations of the Mayer--Vietoris property and K\"unneth formulas, respectively, for elliptic cohomology. See Theorems~\ref{MVK} and \ref{Kun}.

The \emph{covariant} functoriality of $\YY \mapsto \hh(\YY)$ accounts for the \emph{contravariant} functoriality (pull-back maps) of in elliptic cohomology. In order to define Gysin maps or push-forwards (e.g. for smooth proper morphisms), we need to incorporate certain twists\footnote{called Thom sheaves in \cite{GKV}}. These will be encoded by natural functors
$$\Theta_{\YY}:\QC\big(\YY\big)\longrightarrow\QC^{E}\big(\hh(\YY)\big),$$ where the superscript $^E$ indicates $E$-equivariant objects, or more precisely by the determinants
$$\vartheta(-) := \operatorname{det}\Theta(-).$$
For smooth proper $f:\YY\rightarrow\XX$ we will then construct a canonical map of quasi-coherent sheaves $$L_{E}(f)_{*}\vartheta_{\YY}(\mathbf{L}_{f})\longrightarrow\OO_{\hh(\XX)}$$ where $\mathbf{L}_f$ denotes the relative cotangent complex.

\subsubsection{Construction}

Let $E$ be an elliptic curve over $k$. In order to define the \emph{elliptic loop space} $L_{E}\YY$, $1$-categorical Cartier duality suggests that we look for an object $\SSS$ with a symmetric monoidal equivalence $$\big(\!\operatorname{Coh}(\SSS),\otimes\big)\simeq\big(\!\operatorname{Coh}_{0}(E),\star\,\big),$$ and then define $$L_{E}(\YY):=\Map(\SSS,\YY).$$ The hypothetical space $\SSS$ may be viewed as an avatar of $E$ which is \emph{small} enough to ensure the Zariski-locality of $\YY \mapsto L_E(\YY)$. Note that this fails for a less exotic construction like $\Map(E,\YY)$.

Adopting a Tannakian perspective, we bypass the question of defining the space $\SSS$ and instead work with the tensor category \[ \big(\! \operatorname{Coh}(\SSS),\otimes \big) := \big(\!\operatorname{Coh}_{0}(E),\star\big), \]
or rather its ind-completion\footnote{obtained by formally adjoining filtered colimits}, which we denote by $\HE$. We then define $L_{E}\YY$ to classify tensor functors $$\big(\!\QC(\YY),\otimes\big)\rightarrow \HE.$$ When $\YY$ is \emph{Tannakian}, i.e. satisfies the main result of \cite{BHL}, the analogous construction for $\mathbf{G}_a$ or $\GM$ in place of $E$ recovers the rational and trigonometric loop spaces. The natural action of $E$ on $\HE$ accounts for the $E$-action on $\hh(\YY)$.

When $\YY$ is suitably nice (e.g. a global quotient), we will see that $\hh(\YY)$ can be realized as an open \[\hh(\YY) \subset \Map(E,\YY),\] invariant under the $E$-action (where $E$ acts on itself by translation).
Namely, a morphism $S\rightarrow L_{E}\YY$ is the same thing as an $S$-family of maps $$f:S\times E\rightarrow\YY$$ with the property that for all $s\in S$ and $F\in\Perf(\YY)$, the Fourier--Mukai transform of $f_{s}^{*}F \in \Perf(E)$ has $0$-dimensional support.

The twist functor $\Theta_\YY$ is defined as follows. Given $F\in\QC(Y)$ and a point $\phi$ of $L_{E}(\YY)$, $\Theta_{\YY}(F) \in \QC^E(\hh(\YY))$ has fibre at $\phi$ given by $$\phi^{*}\Theta_{\YY}(F):=\operatorname{Hom}_{E}(\mathcal{O}_{e},\phi(F)),$$ with $e$ the neutral element.

\subsubsection{Example: schemes}

Let us describe the elliptic loop space $\hh(\YY)$ in the case of $\YY=Y$ a scheme. By the Zariski codescent property, it will suffice to consider the affine case $Y=\spc(R)$.

By definition, $\hh(Y)$ classifies tensor functors $$\big(\!\QC(Y),\otimes\big)\longrightarrow\big(\HE,\star\big).$$ Since $\QC(Y)$ is generated as a tensor category by the structure sheaf $\OO_Y$, which must map to the monoidal unit $\OO_{e}$ of $\HE$ (where $e$ is the neutral element), such a tensor functor is specified entirely by a $k$-algebra map $$R\longrightarrow\operatorname{End}(\mathcal{O}_{e})\simeq k[\eta]$$ with $\eta$ in homological degree $+1$. Such are classified by the $(-1)$-shifted tangent bundle $\TT[-1]Y$, see \cite{BFN}.

The $E$-action is specified by a comodule structure on $\OO(\TT[-1]Y)$ for the coalgebra $\Gamma(E,\OO)$, which by a standard application of Koszul duality amounts to an derivation of homological degree $-1$ on $\TT[-1]Y$, which is the de Rham differential.

Thus we have a (non-canonical) isomorphism $$\hh(X)\simeq\TT[-1]X.$$ In particular, the loop spaces \[ L^{\mathrm{rat}}(X) \simeq L^{\mathrm{trig}}(X) \simeq \hh(X) \] all agree in the case of schemes.

\subsubsection{Variation of the elliptic curve}

A novel aspect of the elliptic theory, in comparison to the rational and trigonometric variants, is the possibility of varying the elliptic curve $E$. This makes the construction $\hh(X)$ richer than $L^{\mathrm{rat}}(X)$ and $L^{\mathrm{trig}}(X)$ even in the case of schemes where all three are abstractly isomorphic.

First, $\hh(\YY)$ is functorial with respect to isogenies of elliptic curves. Moreover, given any \emph{relative} elliptic curve $E \to S$, there is a variant which takes a stack $\YY$ over $S$ and outputs a derived stack $L_{E/S}(\YY)$ over $S$ (with an $E$-action). For example, this may be applied to the universal elliptic curve $\mathcal{E}$ over the moduli stack $\mathcal{M}_{\operatorname{ell}}$ classifying smooth elliptic curves.

Therefore, even for schemes $\YY = X$ the elliptic theory is richer than the rational and trigonometric versions. As we have seen above, there exists an isomorphism $L_{E}X\simeq\TT[-1]X$ depending on a choice of trivialization of $$H^{1}(E,\OO)\simeq H^{0}(E,\Omega)^{*}\simeq\omega^{*}_{\operatorname{ell}}|_{\{E\}},$$ where $\omega_{\operatorname{ell}} \to \mathcal{M}_{\operatorname{ell}}$ is the Hodge bundle. Said differently, the (fibrewise) affinization of $\mathcal{E} \to \mathcal{M}_{\operatorname{ell}}$ is the non-trivial family of copies $B\mathbf{G}_{a}$, obtained from $\omega_{\operatorname{ell}}^{*}$ via the action of $\GM$ on $B\mathbf{G}_{a}$.

\subsubsection{Example: classifying stacks}

Let us now take $\YY$ to be the classifying stack $BGL_r$ of a general linear group. Recall that there are equivalences $$ L^{\operatorname{rat}}\big(BGL_{r}\big)\simeq \mathfrak{gl}_{r}/GL_{r}, \quad L\big(BGL_{r}\big)\simeq GL_{r}/GL_{r},$$ where the stack quotients are with respect to the adjoint action; see \cite{BNloop}.\footnote{%
  In \cite{BNloop}, the authors consider the \emph{unipotent loop space} $L^{\mathrm{uni}}(\YY) := \Map(B\mathbf{G}_a, \YY)$ rather than the rational loop space $\Map(B\widehat{\mathbf{G}}_a, \YY)$. From our perspective, it is the rational loop space that geometrizes de Rham cohomology for stacks. Indeed, in the above example we have $L^{\operatorname{uni}}\big(BGL_{r}\big) \simeq \mathcal{N}/GL_{r}$, where $\mathcal{N}$ denotes the cone of unipotent elements, which is distinct from the $(-1)$-shifted tangent $\TT[-1]BGL_r \simeq \mathfrak{gl}_{r}/GL_{r}$.
}
We can alternatively realize these as moduli stacks: $$\mathfrak{gl}_{r}/GL_{r}\simeq \{\textrm{length}\, r\,\textrm{sheaves}\,\textrm{on}\,\mathbf{G}_{a}\},$$ $$GL_{r}/GL_{r}\simeq\{\operatorname{length}\, r\,\textrm{sheaves}\,\textrm{on}\,\mathbf{G}_{m}\}.$$ 

For the elliptic counterpart, we claim $$L_{E}\big(BGL_{r}\big) \simeq \{\textrm{length}\, r\,\textrm{objects}\,\textrm{of}\,\HE\} \simeq \operatorname{Bun}_{r}^{\textrm{ss},0}(E),$$ where the equivalence with the moduli stack of semistable vector bundles on $E$ of rank $r$ and degree $0$ is due to Atiyah (see \cite{At}).

We first consider the case $r=1$. By definition, $$L_{E}(B\mathbf{G}_{m})=\Fun(\QC(B\mathbf{G}_{m}),\HE).$$ Tensor functors out of $\QC(B\GM)$ are entirely by their values on the tensor generator $\mathcal{O}(1)$. The image is invertible, hence given by some skyscraper sheaf $\OO_{x}$, which we note has a $\GM$'s worth of automorphisms. We find thus $$\hh(B\mathbf{G}_{m})\simeq E\times B\GM.$$
The twists in this case are given by $$\Theta_{B\GM}(\mathcal{O}(d))\simeq \OO_{\{\operatorname{d-torsion}\,\operatorname{points}\}}\otimes\mathbf{R}\Gamma(E,\mathcal{O}).$$

Similarly, a tensor functor out of $\QC(BGL_{r})$ is specified\footnote{%
  In fact, $\QC(BGL_{r})$ is the universal tensor category containing an object $x$ with $\wedge^{r}x$ invertible and $\wedge^{r+1}x\simeq 0$, by a result of Iwanari (see \cite{Iwa}).
} by the image of the standard representation $V$, subject to invertibility of $$\operatorname{det}V=\bigwedge^{r}V.$$ This invertibility accounts for the length $r$ condition on the image of $V$. (Alternatively, note that the trace $r$ of $V$ must be preserved by a tensor functor.) 

We remark that the canonical map $\hh(BB)\rightarrow \hh(BGL_{r})$ obtained by functoriality from the universal flag bundle $BB\rightarrow BGL_{r}$ recovers the \emph{elliptic Springer resolution} of \cite{BNEll}. This is compatible with the corresponding facts for rational, unipotent, and trigonometric/free loops.

\subsubsection{Tensor categories}

Since our construction of elliptic loop spaces is internal to the world of tensor categories, it is clear that we may contemplate the following generalizations of $\hh(-)$:

First, instead of a stack $\YY$, we can more generally take as input any tensor category $C$. The elliptic loop space $$\hh(C):=\Fun(C,\HE)$$ is then defined to classify tensor functors $$C \rightarrow \HE.$$

In another direction, we may also contemplate modifying the tensor category $\HE$ to get different flavours of loop spaces. As we have already seen, we may take the convolution categories of $0$-dimensionally supported sheaves on $\mathbf{G}_a$ or $\GM$, or more generally the convolution category $\HA$ of $0$-dimensionally supported sheaves on any abelian group scheme $A$. For higher dimensional abelian varieties, this would reproduce \emph{iterated} loop spaces in the case of schemes.

We may even consider a general tensor category $\mathcal{H}$ that is not necessarily of the form $\HA$. In order to obtain a theory analogous to classical loop spaces, we should impose the following properties:
\begin{enumerate}
  \item The centre of $\mathcal{H}$ satisfies $$Z(\mathcal{H}):= \operatorname{End}_{\mathcal{H}}(\mathbf{1})\simeq k\oplus k[-1].$$ This ensures that on affine schemes $S$, we have $$\Fun(\QC(S), \mathcal{H}) \simeq \TT[-1]S.$$
  \item $\mathcal{H}$ is an integral domain: the monoidal product of non-zero objects is non-zero. This means that $\mathcal{H}$ behaves like a ``point'', and guarantees Zariski-locality of $S \mapsto \Fun(\QC(S), \mathcal{H})$.
\end{enumerate}
As one intriguing example we could take the (non-algebraic) group stack $\mathbf{G}_{m}/q^{\mathbf{Z}}$, and consider the convolution category of coherent sheaves whose pull-back to $\mathbf{G}_{m}$ has $0$-dimensional support. Note that sheaves on $\mathbf{G}_{m}/q^{\mathbf{Z}}$ are left modules for the quantum torus $$A_{q}:=k\langle x^{\pm},y^{\pm}\rangle /\big(xy=qyx\big).$$

\subsection{Equivariant elliptic Hodge cohomology}

Let $\YY = X/G$ be a global quotient stack, where $G$ is an algebraic group acting linearly on a quasi-projective scheme $X$. From our point of view the \emph{$G$-equivariant elliptic cohomology} of $X$ can be completely extracted from the structure sheaf of the elliptic loop space $$\hh^G(X) := \hh(X/G),$$ together with the loop rotation action of $E$. More precisely, by analogy with the trigonometric case, we need to apply a Tate construction with respect to the $E$-action, for it is only after passing from Hochschild homology to periodic cyclic homology that we have $\mathbf{A}^1$-homotopy invariance.

In this first paper we content ourselves with introducing what we call the \emph{equivariant elliptic Hodge cohomology}, the cohomology theory obtained from the structure sheaf of $\hh^G(X)$ (without any Tate construction); discussion of equivariant elliptic cohomology itself will be postponed to the sequel of this paper. This is an assignment $$ X \mapsto HH^{G}_{E}(X) \in \QC^E\!\big(\hh^G(*)\big),$$ equipped with various natural functorialities in $X$, $E$, and $G$. By definition, we set \[ HH^G_E(X) := \hh^G(\pi_X)_* \big( \OO_{\hh^G(X)} \big), \] where $\pi_X : X/G \to BG$ is the projection. The natural functorialities of $\hh^G(-)$ in $X$, $E$, and $G$, descend to $HH^G_E(-)$.

As we have explained above, the target of $HH^G(-)$ can equivalently be taken to be $E$-equivariant quasi-coherent sheaves on the moduli stack  $\operatorname{Bun}^{\operatorname{ss,0}}_{G}(E)$ of semistable $G$-bundles of degree $0$ on $E$. Note that the ``taking supports'' map $$\hh^{GL_r}(*) \simeq \operatorname{Bun}_{r}^{\textrm{ss},0}(E) \to E^{(r)}$$ exhibits the symmetric product $E^{(r)}$ as a good moduli space. That equivariant elliptic cohomology should live on the symmetric product $E^{(r)}$ is compatible with the axiomatics recorded in \cite{GKV}.

As an illustration, consider the quotient stack $\PP/\GM$ where $\GM$ acts with weight $d$. Consider the composite $$\pi:L_{E}(\PP/\GM)\rightarrow L_{E}(B\GM)\simeq E\times B\GM\rightarrow E.$$ It can be shown that $\pi_{*}\OO\simeq\mathcal{O}\oplus\vartheta^{\otimes -d^{2}}$, where $\vartheta:=\mathcal{O}(e)$. (Note that the projection $E\times B\GM\rightarrow E$ is a good moduli space.)
The reader familiar with torus-equivariant elliptic cohomology will recognize this computation for $\mathbf{P}^{1}$. That is, the Hodge version of equivariant elliptic cohomology is already the correct answer in this case. Informally speaking, this is due to the fact that Hodge to de Rham degenerates in this case, so there is no need to pass to Tate constructions.

\subsection{Other perspectives on elliptic cohomology}
\label{ssec:other}

The primary motivation for our work is geometric representation theory, where it is often of great interest to study elliptic counterparts of various algebraic structures constructed using cohomology and K-theory. We refer for example to the theories of elliptic quantum groups, elliptic cohomological Hall algebras, and elliptic stable envelopes (see e.g. \cite{GKV,AO,YangZhao1,YangZhao2} and \cite[\S 5.2]{Okounkov} for an overview). These constructions are typically based on a set of axioms sketched by Ginzburg--Kapranov--Vasserot \cite{GKV} as desiderata for equivariant elliptic cohomology. The present paper is the first step in our approach to a canonical construction satisfying these axioms.

There are several existing approaches to equivariant elliptic cohomology. The first construction, due to Grojnowski (see \cite{Gr}), works with complex coefficients and is complex-analytic in nature. He defines a $G$-equivariant theory for $G=T$ a torus, and for general $G$ by taking Weyl invariants of the $T$-equivariant theory for $T$ a maximal torus. However, for $G=e$ trivial, the theory degenerates to ordinary cohomology by definition, rather than the non-equivariant elliptic cohomology theory of Landweber \cite{Landweber} defined using techniques of stable homotopy theory and homotopical algebra.

Another perspective is taken in the work of J.~Lurie \cite{LurieICM}, and developed further by Gepner--Meier \cite{GepMei} in order to define an equivariant elliptic cohomology theory for any (oriented) elliptic curve $E$ over an $E_\infty$-ring spectrum and any compact Lie group $G$. However, as in Grojnowski's work, the case of general $G$ is defined by formal extending the theory from the abelian case.

In order to explain the expected connection with \cite{GepMei}, let us first point out that all our constructions and main results go through for an elliptic curve not just over a characteristic zero\footnote{The characteristic zero hypothesis is only used in a few places, e.g. to guarantee the linear reductivity of general linear groups $GL_r$.} field $k$, but for any connective $E_\infty$-ring spectrum (e.g. the sphere spectrum).\footnote{We do not pursue it in this text, due to the additional complexity of working in the context of spectral algebraic geometry. Since our motivation comes from geometric representation theory rather than homotopy theory, we prefer not to develop this connection explicitly here.} In particular, the constructions of this paper (and the sequel) will produce for any spectral elliptic curve $E$ a canonical $G$-equivariant elliptic cohomology theory in Lurie's sense, directly and uniformly for any $G$, and purely in algebro-geometric terms.

Lastly we must mention the beautiful work of Sibilla and Tomasini in \cite{ST}, who proposed a new approach to torus-equivariant elliptic cohomology. Their work is based on an open substack of ``quasi-constant'' maps inside $\Map(E,X/T)$, which they showed recovers Grojnowski's construction. Our work is directly motivated by trying to generalize this type of construction to a general $G$. For $G=T$ a torus, it is not difficult to show that points of our elliptic loop space $\hh(X/T)$ can be identified with quasi-constant maps. However, the quasi-constancy condition does not seem to be the relevant one for a general algebraic group $G$. As an illustration, let us remark that it is not the case that the stack of quasi-constant maps into $BG$ is covered by the stack of quasi-constant maps into $BB$, whereas this is the case for $L_{E}(BB)$ (see Lemma~\ref{springer}).

\subsection{Acknowledgments}

The authors would like to thank Andrei Okounkov for the suggestion to look into these questions, and for comments on a draft.
E.B. would like to thank Wille Liu, Andy Jiang, Ian Grojnowski and Arkadij Bojko for helpful conversations related to this material.
A.A.K. acknowledges support from the grants AS-CDA-112-M01 and NSTC 112-2628-M-001-0062030.

\section{Stacks of tensor functors}

\subsection{Recollections on derived geometry}We saw in the introduction that even in the simplest case of an affine scheme $X$, our construction of $\hh(X)$ is forced to pass through derived schemes. Ultimately this complexity bleeds into the entire theory, and we must work throughout in a derived context in order to have a clean formulation. Our references for this section are the books \cite{GR}, \cite{Lu} and \cite{HAG}.

Henceforth everything is derived and $\infty$-categorical by assumption, with the prefixes \emph{derived} and $\infty$-dropped. We will use \emph{homological} grading conventions throughout.

Let $k$ be a commutative ring containing $\mathbf{Q}$.\footnote{As mentioned in the introduction, this assumption can be dropped almost everywhere, with a few obvious exceptions (e.g. the computation of the rational loop space as the $(-1)$-shifted tangent). We can also work over an $E_\infty$-ring spectrum, using Lurie's notion of spectral elliptic curves, with obvious modifications that we leave to the interested reader.}
We let $\dAlg$ denote the category of commutative $k$-algebras. Since we are in characteristic zero, these can be modelled equivalently by connective commutative dg-algebras, simplicial commutative algebras, or connective $E_\infty$-algebras. A \emph{pre-stack} is by definition a functor: $$\YY: R \longmapsto \YY(R),$$ encoding for each $R \in \dAlg$ a homotopy type (or $\infty$-groupoid) $\YY(R)$ of $R$-points.
A \emph{stack} is a pre-stack satisfying descent for the \'{e}tale topology. The category of such is denoted $\mathbf{St}_{k}$. Representable objects are called \emph{affine schemes}.
Recall that there is a notion of (higher) \emph{Artin} stacks (sometimes called \emph{geometric} stacks), see e.g. \cite[Chap.~2, \S 4]{GR}.

For a commutative $k$-algebra $R$ we write $\Mod_R$ for the category of $R$-modules.\footnote{Note that our conventions are nonstandard here, in that objects of $\Mod_R$ are implicitly connective. The same goes for $\QC(\YY)$. We write $\DD(R)$ and $\DD(\YY)$ for the categories of nonconnective objects.} These can be modelled equivalently by connective dg-modules, simplicial modules, or connective module spectra. For a stack $\YY$, the category of quasi-coherent sheaves $\QC(\YY)$ is defined by $\QC(\spc(R)) = \Mod_R$ in the affine case, and in general by the limit \[\QC(\YY) = \lim_{R,y} \Mod_R \] over $R\in\dAlg$ and $R$-points $y \in \YY(R)$. The category $\QC(\YY)$ is presentable, and there is a unique symmetric monoidal structure on $\QC(\YY)$ which is colimit-preserving in each argument and for which the pull-backs $\QC(\YY) \to \Mod(R)$ are symmetric monoidal for every $R\in\dAlg$ and every morphism $\spc(R) \to \YY$.
The categories $\Mod_R$ and $\QC(\YY)$ are prestable in the sense of Lurie (see \cite[App.~C]{SAG}), and form the connective parts of $t$-structures on their stabilizations $\DD(R)$ and $\DD(\YY)$.

We write $\Mod^{\mathrm{perf}}_R \subset \Mod_R$ for the full subcategory of \emph{perfect} $R$-modules, generated by $R$ under finite colimits and extensions. We write $\Perf(\YY) \subset \QC(\YY)$ for the full subcategory of \emph{perfect complexes} on $\YY$, i.e. the limit of $\Mod_R^{\mathrm{perf}}$ over $R\in\dAlg$ and $y \in \YY(R)$.

When $\YY$ is Artin, $\QC(\YY)$ and $\Perf(\YY)$ can be described as the respective limits of $\Mod_R$ and $\Mod_R^{\mathrm{perf}}$ over pairs $(R,y)$ where $R\in\dAlg$ and $y : \spc(R) \to \YY$ is a \emph{smooth} morphism. It follows from \cite[Prop.~C.3.2.4]{SAG} that $\QC(\YY)$ is a Grothendieck prestable category in the sense of \cite[\S C.1.4]{SAG}.

Finally, for any stack $\YY$, there is an associated \emph{affine stack} $$\YY\rightarrow\operatorname{Aff}(\YY)$$ such that any morphism $\YY\rightarrow S$ with $S$ affine factors uniquely through $\operatorname{Aff}(\YY)$. For example, the affinization of an elliptic curve $E$ is non-canonically isomorphic to $B\mathbf{G}_{a}$, with an isomorphism corresponding to a trivialization of the one dimensional vector space $H^{1}(E,\mathcal{O})$. For a reference on affine stacks and affinization see \cite[\S 3]{BNloop}.

\subsection{Tensor categories} \label{ssec:tensorcat}
We introduce now the categorical backdrop for our main constructions. We refer again to \cite[\S C.1.4]{SAG} for background on Lurie's theory of Grothendieck prestable $\infty$-categories.

A \emph{tensor category} is a symmetric monoidal Grothendieck prestable $k$-linear category, where by convention the tensor product is assumed colimit-preserving in each argument. Given tensor categories $C$ and $D$, we denote by $\FUN(C,D)$ the category of \emph{tensor functors}, i.e. colimit-preserving symmetric monoidal $k$-linear functors, from $C$ to $D$. We denote by $\CCC$ the category of tensor categories and tensor functors. This is in other words the category of commutative algebra objects in Lurie's symmetric monoidal category of Grothendieck prestable $k$-linear categories, see \cite[\S C.4.2]{SAG}.

\begin{example} For an Artin stack $\YY$, $\QC(\YY)$ is a tensor category. Further, for a morphism $f:\YY\rightarrow\XX$ the induced functor $$f^{*}:\QC(\XX)\longrightarrow\QC(\YY)$$ is a tensor functor. \end{example}

\begin{example} Let us say a \emph{tensor $1$-category} is a Grothendieck abelian $1$-category equipped with a symmetric monoidal structure for which the tensor product commutes with colimits in each argument, and a \emph{tensor functor} between tensor $1$-categories is a colimit-preserving symmetric monoidal functor. Given a ``nice'' tensor $1$-category $A$\footnote{%
  More precisely, assume that $A$ admits a ``weakly flat descent structure'' in the sense of \cite[2.2, 3.1]{CisDeg}.
  For our purposes it is enough to consider the case of $A$ generated by compact projective objects, in which case $A$ is equivalent to the category of functors $(A_0)^{\mathrm{op}} \to \mathrm{Ab}$ that preserve finite products, where $A_0$ is the full subcategory of compact projectives (see e.g. \cite[\S 5.1]{CesScholze}).
}, the connective derived category $\DD_{\geq 0}(A)$ is a tensor category, and for any tensor functor $f : A \to B$, the left-derived functor $$Lf : \DD_{\geq 0}(A) \to \DD_{\geq 0}(B)$$ is a tensor functor. See \cite[Thm.~2.14, Prop.~3.2]{CisDeg} and \cite[Prop.~1.3.5.15, Cor.~1.3.4.26, Ex.~4.1.7.6]{Lu}. \end{example}

\subsection{Stacks of tensor functors} \label{ssec:stktenfun}
Recall that for two stacks $\XX$ and $\YY$, there is a \emph{mapping stack}
\[\Map(\XX,\YY) : R \mapsto \Map(\XX \otimes R, \YY),\]
where we write $\XX \otimes R := \XX \times \spc(R)$.
We introduce an analogue of this construction for tensor categories.

Given a tensor category $D$, we may consider the tensor category $$D\otimes R:=D\otimes\Mod_{R}$$ for every commutative algebra $R \in \dAlg$. Recall that this is defined by the Lurie tensor product, see \cite[\S C.4.2]{SAG}.

\begin{definition} Let $C$ and $D$ be tensor categories. The \emph{stack of tensor functors} from $C$ to $D$ is defined by $$\Fun(C,D) : R \mapsto \FUN(C,D\otimes R)^{\simeq}.$$ \end{definition}

\subsubsection{Tannakian stacks}

There is a canonical morphism of stacks $$\Map(\XX,\YY)\rightarrow\Fun(\QC(\YY),\QC(\XX)),$$ sending $f \mapsto f^*$. We call $\YY$ \emph{Tannakian} when this map is invertible for all stacks $\XX$ (or equivalently, for all affine schemes $\XX=S$).

For example, when $\YY$ has quasi-affine diagonal and has compactly generated derived category $\DD(\YY)$, it is Tannakian by \cite[Thm.~1.3]{BHL}.\footnote{Note that the compact generation is known for a large class of algebraic stacks by the work of many authors, see \cite[\S 1.8]{KhanKstack} for an overview.}

\subsubsection{The coevaluation functor}\label{sssec:coev}

Given tensor categories $C$ and $D$, consider the functor \[ \YY \mapsto \FUN(C, D \otimes \YY)^\simeq, \] where $D\otimes\YY:=D\otimes\QC(\YY)$. This satisfies étale descent on the category of stacks, since $\QC(-)$ does. In particular, for any stack $\YY$, we have
\[ \homs(\YY, \Fun(C,D)) \simeq \FUN(C, D \otimes \YY)^\simeq. \]
In particular, the identity $$\operatorname{id}:\Fun(C,D)\rightarrow\Fun(C,D)$$ classifies a universal functor $$\mathrm{coev}_{C}:C\rightarrow D \otimes \QC(\Fun(C,D)).$$

\subsubsection{Pull-backs of functor stacks}

\begin{lemma}\label{fibre} Given a diagram of tensor functors $$C_{0}\leftarrow C_{01}\rightarrow C_{1},$$ the induced square \[\begin{tikzcd}
\Fun(C_{0}\otimes_{C_{01}}C_{1},D)\arrow{r} \arrow[swap]{d} & \Fun(C_{0},D) \arrow{d} \\
\Fun(C_{1},D)\arrow{r} & \Fun(C_{01},D)
\end{tikzcd}\] is cartesian for every tensor category $D$. \end{lemma}

\begin{proof} Note that $\CCC$ is the category of commutative algebra objects in the category of Grothendieck prestable categories and colimit-preserving functors.  By \cite[Prop.~3.2.4.7]{Lu}, the tensor product $C_{0}\otimes_{C_{01}}C_{1}$ is thus the coproduct in $\CCC$ of the diagram $C_{0}\leftarrow C_{01}\rightarrow C_{1}$. The claim follows. \end{proof}

\subsubsection{Tensor $1$-categories}

Recall from Subsect.~\ref{ssec:tensorcat} that for tensor $1$-categories $A$ and $B$, we have the category $$\FUN(A, B)$$ of tensor functors $A \to B$. For an ordinary commutative $k$-algebra $R$, we denote by $B \otimes R := B \otimes \Mod^\heartsuit_R$ the Lurie tensor product with the tensor $1$-category of ordinary $R$-modules; this is still a tensor $1$-category by \cite[Thm.~C.5.4.16]{SAG}. The classical stack of tensor functors is defined by
$$\Fun(A,B) : R \mapsto \FUN(A, B\otimes R)^\simeq.$$
We also have the classical stack $\Fun(\DD_{\geq 0}(A),\DD_{\geq 0}(B))^{\cl}$, the classical truncation of the stack of tensor functors $\DD_{\geq 0}(A) \to \DD_{\geq 0}(B)$.
There is a canonical morphism $$\Fun(A,B)\longrightarrow\Fun(\DD_{\geq 0}(A),\DD_{\geq 0}(B))^{\cl},$$ sending a tensor functor $f : A \to B$ to the induced left-derived functor $Lf : \DD_{\geq 0}(A) \to \DD_{\geq 0}(B)$ on connective derived categories.
More precisely, on $R$-points it sends a tensor functor $f : A \to B \otimes \Mod^\heartsuit_R$ to $Lf : \DD_{\geq 0}(A) \to \DD_{\geq 0}(B) \otimes \Mod_R$, for every commutative $k$-algebra $R$.

\begin{lemma}\label{abelian} Let $A$ be a tensor $1$-category which is generated under colimits by a set of objects which are compact, projective, and dualizable. Then the canonical morphism $$\Fun(A,B)\longrightarrow\Fun(\DD_{\geq 0}(A),\DD_{\geq 0}(B))^{\cl}$$ is invertible. \end{lemma}

\begin{proof} The claim is that for every ordinary commutative $k$-algebra $R$, the functor
\[ \FUN(A, B \otimes \Mod^\heartsuit_R)^\simeq \to \FUN(\DD_{\geq 0}(A), \DD_{\geq 0}(B) \otimes \Mod_R)^\simeq \]
is an equivalence. We will show that, under the given assumptions, it is inverse to the functor sending a tensor functor $F$ to the tensor functor $X \mapsto \operatorname{H}_0(F(X))$.

Denote by $A_0 \subset A$ the full subcategory spanned by finite coproducts of the given compact projective dualizable generators.
The assumptions imply that $A$ is freely generated by $A_0$ under sifted colimits, in the $1$-categorical sense, while the same holds for $\DD_{\geq 0}(A)$ but in the $\infty$-categorical sense (see e.g. \cite[\S 5.1]{CesScholze}). Let $F : \DD_{\geq 0}(A) \to \DD_{\geq 0}(B) \otimes \Mod_R$ be a tensor functor. We claim that its restriction $F|_{A_0}$ takes values in $B \otimes \Mod^\heartsuit_R$ (which is the full subcategory of discrete objects in $\DD_{\geq 0}(B) \otimes \Mod_R$, see the proof of \cite[Thm.~C.5.4.16]{SAG}). If $X$ is a dualizable object of $A$, then the functor $X \otimes (-)$ is right adjoint to $X^\vee \otimes (-)$ and is in particular left-exact, so that $X$ is flat. In particular, $X \otimes^L (-) \simeq X \otimes (-)$ so that $X$ remains dualizable in $\DD_{\geq 0}(A)$. Since $F : \DD_{\geq 0}(A) \to \DD_{\geq 0}(B) \otimes \Mod_R$ was symmetric monoidal, it follows that $F(X)$ is also dualizable and hence flat by the same argument as above. In particular, $F(X)$ must be discrete. Thus, $F$ restricts to the functor 
\[ F : A_0 \to B \otimes \Mod^\heartsuit_R. \]
We let $f : A \to B \otimes \Mod^\heartsuit_R$ be its unique colimit-preserving extension to $A$. The assignment $F \mapsto f$ is clearly natural in $F$, inverse to $f \mapsto Lf$, and satisfies $\operatorname{H}_0 (F(X)) \simeq \operatorname{H}_0 (Lf(X)) \simeq f(X)$ for every $X$ in $A$. The claim follows. \end{proof}

\subsubsection{Tensor localizations} We are interested in computing $\Fun(C,D)$ \emph{locally} on $C$. We formalize this as follows:

\begin{definition} If $C$ and $C'$ are tensor categories and $$j^{*}:C\rightarrow C'$$ is a tensor functor admitting a fully faithful right adjoint $j_{*}$, we call $j^{*}$ a \emph{tensor localization}.\footnote{The adjoint $j_{*}$ is of course not assumed to be symmetric monoidal.}\end{definition}

\begin{definition}\label{cover} Let $C$ be a tensor category. A (finite) \emph{cover} of $C$ is a (finite) jointly conservative family of tensor localizations $$\{j^{*}_{\alpha}:C\rightarrow C_{\alpha}\}_{\alpha\in A}.$$ \end{definition}

\begin{lemma}\label{Cart} Let $j^{*}:C\rightarrow C'$ be a tensor localization. Then for any fully faithful tensor functor $\iota:D\rightarrow D'$, the commutative square of stacks \[\begin{tikzcd}
\Fun(C',D)\arrow{r} \arrow[swap]{d} & \Fun(C',D') \arrow{d} \\
\Fun(C,D)\arrow{r} & \Fun(C,D')
\end{tikzcd}\] is cartesian. \end{lemma}

\begin{proof} Replacing $D$ and $D'$ by $D \otimes R$ and $D' \otimes R$ for all $R\in\dAlg$, it will suffice to show that the square of groupoids \[\begin{tikzcd}
\FUN(C',D)^\simeq\arrow{r} \arrow[swap]{d} & \FUN(C',D')^\simeq \arrow{d} \\
\FUN(C,D)^\simeq\arrow{r} & \FUN(C,D')^\simeq
\end{tikzcd}\] is cartesian. Since $\iota$ is fully faithful, the horizontal arrows are injective on $\pi_0$ and bijective on all higher homotopy groups. It will thus suffice to show that the functor \[ \FUN(C',D)^\simeq \to \FUN(C,D)^\simeq \times_{\FUN(C,D')^\simeq} \FUN(C',D')^\simeq \] is essentially surjective and bijective on connected components of mapping groupoids. This amounts to the assertion that for any commutative square of solid arrows \[\begin{tikzcd}
C \arrow{r}{j^{*}} \arrow[swap]{d}{\phi} & C' \arrow{d}{\psi}\arrow[dashed,swap]{ld}{\chi} \\
D\arrow{r}{\iota} & D',
\end{tikzcd}\]
the dashed arrow can be filled in by an essentially unique tensor functor $\chi:C'\rightarrow D$. In fact, any such $\chi$ will automatically be a tensor functor: since $\iota$ is fully faithful, colimit-preserving and symmetric monoidal, this follows from the fact that $\iota \chi \simeq \psi j^*$ is a tensor functor.

First note that if $\chi$ is such a filling, then the commutativity $\chi j^* \simeq \phi$ and the adjunction counit provide a canonical isomorphism $\phi j_* \simeq \chi j^* j_* \simeq \chi$. It will thus suffice to show that $\chi := \phi j_{*} : C' \to D$ indeed fills the above diagram. The counit of the adjunction provides a natural transformation $$\iota\chi = \iota \phi j_{*} \simeq \psi j^{*} j_{*} \xrightarrow{\sim} \psi$$ is invertible because $j_{*}$ is fully faithful. Similarly, we claim that the natural transformation
\[ \phi \to \phi j_{*} j^* = \chi j^* \]
provided by the unit is also invertible. This may be checked after applying $\iota$ on the left: since $\iota$ is fully faithful and colimit-preserving, it admits a right adjoint $\iota^R$ with $\iota^R \iota \simeq \id$. Under the identification $\iota \phi \simeq \psi j^*$, $\iota \phi \to \iota \phi j_{*} j^*$ is inverse to the counit isomorphism $\psi j^* j_{*} j^* \to \psi j^*$. \end{proof}

\subsubsection{Integral domains}

We will now introduce a very special class of tensor categories.
As we will see, these behave like ``points'' in the world of tensor categories.

\begin{definition} We say that the tensor category $D$ is an \emph{integral domain} if for non-zero objects $d_{0},d_{1}$ it holds that $d_{0}\otimes d_{1}$ is also non-zero. \end{definition}

\begin{remark} $\Mod_{R}$ is an integral domain iff $R$ is a field. If $\YY$ is a reduced stack whose underlying topological space is a point, then $\QC(\YY)$ is an integral domain. Indeed, there exists a flat surjection $\pi:\spc(\kappa)\rightarrow\YY$ for some field $\kappa$ (see \cite[Tag~06MN]{StacksProject}), and $\pi^{*}$ is a conservative tensor functor to the integral domain $\Mod_{\kappa}$. \end{remark}

\begin{lemma}\label{local} If $D$ is an integral domain, and we are given a finite cover $\{j^{*}_{\alpha}:C\rightarrow C_{\alpha}\}$ of $C$, then any tensor functor $F : C \to D$ necessarily factors through one of the $C_{\alpha}$. Equivalently, there exists an $\alpha$ such that $F_\alpha := F j_{\alpha,*} : C_\alpha \to D$ is a tensor functor and the natural transformation $F \to F j_{\alpha,*}j_{\alpha}^* =F_\alpha j_{\alpha}^*$ is invertible. \end{lemma}
\begin{proof} Since $j_\alpha^*$ is a tensor localization, such a factorization $$F_\alpha : C_\alpha \to D$$ exists (as a tensor functor) if and only if $F_{\alpha} \simeq F j_{\alpha,*}$. We first show the existence of $F_\alpha$ as a colimit-preserving functor, and then check that it is endowed with a symmetric monoidal structure. Let $K_{\alpha}$ denote the kernel of the localization $j_{\alpha}^*$. Assuming that $F$ did not factor through any $j_{\alpha}^*$ we could then pick $d_{\alpha}$ in $K_{\alpha}$ so that $F(d_{\alpha})$ is non-zero. However $\bigotimes_{\alpha\in A}d_{\alpha}\simeq 0,$ as it vanishes upon localization at each fixed $\alpha$ and the cover is by definition jointly conservative. It follows since $F$ is symmetric monoidal that $$\bigotimes_{\alpha\in A}F(d_{\alpha}) \simeq F\big(\bigotimes_{\alpha\in A}d_{\alpha}\big) \simeq 0.$$  We obtain thus a contradiction as $D$ is an integral domain. 

It remains to show that the induced functor $F_{\alpha} : C_{\alpha}\rightarrow D$ is symmetric monoidal. As the right adjoint of a symmetric monoidal functor, $j_{\alpha,*}$ is automatically lax symmetric monoidal and so we have natural maps $\mathbf{1}_{C} \to j_{\alpha,*}(\mathbf{1}_{C_\alpha})$ and $$j_{\alpha,*}(x)\otimes j_{\alpha,*}(y)\rightarrow j_{\alpha,*}(x\otimes y)$$ for objects $x$ and $y$. Since $F$ is symmetric monoidal, it will suffice to show that these maps become invertible after applying $F$.  Note that the cone of any of these maps automatically lies in $K_{\alpha}=\operatorname{ker}(j_{\alpha}^{*})$, as $j_{\alpha}^{*}$ is monoidal and $j_{\alpha}^{*}j_{\alpha,*}\simeq\id$. In particular, it vanishes after applying $F \simeq F_\alpha j_\alpha^*$. The claim follows. \end{proof}

\begin{remark} Let $C$ be an integral domain. Applying Lemma~\ref{local} to the identity functor, we see that for any finite cover $\{j^{*}_{\alpha}:C\rightarrow C_{\alpha}\}$, there exists an $\alpha$ such that $\id \to j_{\alpha,*}j_\alpha^*$ is invertible. But this implies that $j_\alpha^* : C \to C_\alpha$ is an equivalence. In particular, integral domains admit no nontrivial finite covers. More strongly, they admit no non-trivial tensor localizations. Equivalently, if $C$ is an integral domain then $C$ contains no non-trivial $\otimes$-idempotent algebras. In particular the locale formed by idempotent algebras (\emph{cf} \cite{BKS}) is trivial. \end{remark}

\section{The tensor category $\HE$} The key player in this note is a tensor category $\HE$ associated to an elliptic curve $E$ over $k=\mathbf{C}$. In this section we introduce it and discuss some of its basic properties. In fact, the construction of $\HE$ makes sense for general abelian group scheme $A$ in place of $E$, or a relative abelian group scheme $E \to S$.

\subsection{Sheaves with $0$-dimensional support}

We begin with some generalities. Let $\pi : X \to S$ be a schematic morphism of stacks.

\begin{definition}\label{perf0} We denote by $\Perf_{0}(X/S)$ the full subcategory of $\Perf(X)$ spanned by perfect complexes $F$ satisfying the following condition: \begin{enumerate} \item[$(\ast)$] For every field $\kappa$, every $s \in S(\kappa)$, and every affine open $U \subset X_s := X \times_S \{\kappa\}$, the complex of global sections $\Gamma(U, F|_{X_s})$ is a perfect $\kappa$-module. \end{enumerate} We write $\QC_{0}(X/S)$ for the category of ind-objects in $\Perf_{0}(X)$. \end{definition}

\begin{remark} If $\pi : X \to S$ is affine, then a perfect complex $F \in \Perf(X)$ belongs to $\Perf_0(X/S)$ iff $\pi_*(F)$ is perfect. The key observation is that when $S=\spc(\kappa)$, $F \in \Perf(X)$ having $\pi_*(F) \simeq \Gamma(X, F)$ perfect implies that $\Gamma(U, F|_U)$ is also perfect for every standard open $U = \spc(R_f)$ of $X = \spc(R)$. Indeed, since restriction to opens is exact, and over a field every bounded coherent complex is perfect, we may reduce to the case where $R$ is an ordinary algebra and $F$ lies in the heart, in which case the claim follows from \cite[Tag~067E]{StacksProject}. \end{remark}

\begin{example} When $S=\spc(k)$, we write simply $\Perf_{0}(X)$ and $\QC_{0}(X)$. Note that $\QC_0(X)$ is the full subcategory of quasi-coherent sheaves on $X$ whose cohomology sheaves all have $0$-dimensional (or empty) support. \end{example}

\begin{remark} Let $F \in \Perf(X)$ be a perfect complex whose cohomology sheaves have support of fibre dimension $\le 0$ over $S$. Then $F$ lies in $\Perf_0(X/S)$. \end{remark}

\begin{lemma} Let $X$ be a scheme. For every stack $\YY$, there is a natural equivalence $$\Perf_{0}(X)\otimes\Perf(\YY)\rightarrow\Perf_{0}(X\times \YY/\YY).$$ \end{lemma}

\begin{proof}
Pull-backs are easily seen to determine functors $$\Perf_{0}(X)\rightarrow\Perf_{0}(X\times\YY/\YY),\,\,\Perf(\YY)\rightarrow\Perf_{0}(X\times\YY)$$ and the universal property of the tensor product induces a morphism $$\Perf_{0}(X)\otimes\Perf(\YY)\rightarrow\Perf_{0}(X\times\YY).$$ To see that it is an equivalence it suffices to take $X=\spc(R)$ affine, as $\otimes$ commutes with colimits in its arguments.

 Recall that for $X$ affine the right hand side is characterized as those objects of $\Perf(X\times\YY)$ that remain perfect after push-forward to $\YY$. We have a standard equivalence $$\Perf(X\times\YY)\cong\Perf(X)\otimes\Perf(\YY)$$ with respect to which push-forward corresponds to the projection $$\Perf(X)\rightarrow\Mod_{k},$$ and we must show then that the same condition characterizes the left hand side. We can reduce to the consideration of objects $$P=P_{X}\boxtimes P_{\YY}$$  as these generate under finite colimits and retracts. Then by pushing forward along the projection $X\rightarrow *$ the result follows from the following simple observation: 
 if $V\in\Mod_{k}$ is \emph{not} perfect and $P_{\YY}\in\Perf(\YY)$, then $V\otimes P_{\YY}$ is perfect iff $P_{\YY}=0$. 
 Indeed, up to shifts we can assume that $k^{\infty}$ is a retract of $V$ and so $k^{\infty}\otimes P_{\YY}\simeq P_{\YY}^{\infty}$ is a retract of $V\otimes P_{\YY}$ and thus perfect. Now write $P_{\YY}^{\infty}=\operatorname{colim}_{n}P_{\YY}^{n}$ and note that the identity map of $P_{\YY}^{\infty}$ must factor through some $P_{\YY}^{n}$, from which it follows that $P_{\YY}=0.$ \end{proof}

Passing to ind-objects, we get:\footnote{When $\QC(\YY)$ is compactly generated by perfect complexes, we moreover have $\QC(\YY) \simeq \Ind\Perf(\YY).$}

\begin{corollary} Let $X$ be a scheme. For every stack $\YY$, there is a natural equivalence of categories $$\QC_{0}(X)\otimes\Ind\Perf(\YY)\rightarrow\QC_{0}(X\times \YY/\YY).$$ \end{corollary}

\subsection{Convolution tensor categories}

Let $S$ be a stack and $A$ be a relative abelian group scheme over $S$. This means that $A$ is equipped with a multiplication map $m : A \times_S A \to A$, an inversion map $\iota : A \to A$, and a neutral element $e : S \to A$, all satisfying the axioms of an abelian group up to coherent homotopy. We assume that the projection $\pi : A \to S$ is schematic of finite type. We write $[d]:A\rightarrow A$ for scalar multiplication by $d\in\mathbf{Z}$, and $A[d]$ for the finite flat subgroup of $d$-torsion points.

\begin{definition} We define the convolution product on $\QC(A)$, denoted $\star$, by $$F\star G:=m_{*}(F \boxtimes G).$$\end{definition}

\begin{remark} The convolution product is commutative because the multiplication on $A$ is. The unit for $\star$ is the structure sheaf at the origin $\mathcal{O}_{e}$. That this defines a symmetric monoidal structure on $\QC(A)$ in the sense of $\infty$-category theory is explained in \cite[Vol.~I, Chap.~5, \S 5]{GR}. \end{remark}

It follows from the definitions that the convolution product on $\QC(A)$ restricts to $\QC_{0}(A/S)$ and $\Perf_{0}(A/S)$.

\begin{definition} We denote by $\mathcal{H}_{A/S}$ the tensor category $\QC_{0}(A/S)$ equipped with the convolution product $\star$. \end{definition}

\begin{remark} Recall that we denote by $\DD(A)$ the stable category of all quasi-coherent complexes on $A$, so that $\DD(A)$ is equipped with a $t$-structure whose connective part is $\QC(A)$. Just as in Definition~\ref{perf0}, we can define a full subcategory $$\DD_{0}(A/S) \subset \DD(A)$$ such that $\QC_{0}(A/S)$ is the full subcategory of connective objects in $\DD_{0}(A/S)$. Moreover, the convolution product $\star$ extends to a symmetric monoidal structure on $\DD_{0}(A)$. \end{remark}

\subsubsection{The $\check{A}$-action on $\HA$} \label{sssec:actonHA}

\begin{definition} A \emph{multiplicative line bundle} on $A$ is a homomorphism of group stacks $A\rightarrow B\GM$. We denote by
$$\check{A} := \operatorname{GroupMap}(A, B\GM)$$
the stack of multiplicative line bundles on $A$, and refer to it as the \emph{dual group}, with group structure inherited from the group structure on the target. \end{definition}

We may regard this notion as a ``$1$-categorical'' version of a character. Unwinding definitions, a multiplicative line bundle on $A$ amounts to the data of a line bundle $L$ on $A$ together with a \emph{multiplicative structure}, i.e., a distinguished isomorphism $$m^{*}L\rightarrow\big(L\boxtimes L\big)$$ with homotopy coherence data. 

\begin{remark} $\mathcal{O}$ is canonically endowed with a multiplicative structure and the resulting point of $\check{A}$ gives the unit for the group operation. We have $\operatorname{Aut}_{\check{A}}(\mathcal{O})=\chi(A)$, where $$\chi(A):=\operatorname{GroupMap}(A,\GM)$$ is the group of $0$-categorical (i.e. usual) characters. In particular, if there is a unique (up to isomorphism) multiplicative line bundle on $A$, then we have $\check{A}\simeq B\chi(A).$ For example, we have $$\check{\mathbf{G}}_{a}\simeq B\widehat{\mathbf{G}}_{a},\quad\check{\mathbf{G}}_{m}\simeq B\mathbf{Z}=S^{1}.$$\end{remark}

\begin{lemma}\label{action} There is a canonical action of $\check{A}$ on the tensor category $\mathcal{H}_{A/S}$. \end{lemma}

\begin{proof} Let $L$ denote a multiplicative line bundle, and $\operatorname{act}_{L}$ denote the endofunctor of $\mathcal{H}_{A/S}$ defined by $$F\mapsto \operatorname{act}_{L}(F):=L\otimes F.$$ Note that usual tensor product by $L$ preserves the $0$-dimensionality condition. We must show that this morphism respects the monoidal product $\star$ if $L$ is multiplicative. Indeed by definition we have $$\operatorname{act}_{L}(F\star G)=L\otimes m_{*}(F\boxtimes G).$$ The projection formula identifies this with $$m_{*}(m^{*}L\otimes(F\boxtimes G))$$ and multiplicativity of $L$ identifies this latter with $$m_{*}(L\boxtimes L\otimes F\boxtimes G).$$ Re-arranging the tensor factors and using monoidality of $\pi_{i}^{*}$, this identifies with $$\operatorname{act}_{L}(F)\star\operatorname{act}_{L}(G)$$ and we are done. \end{proof}

\subsubsection{Duals and internal homs in $\HA$} The convolution product $\star$ commutes with colimits in each argument separately, and so we have internal hom objects $$\hhoms_{A}^{\star}(F,G) \in \mathcal{H}_{A/S},$$ defined by natural equivalences $$\homs(F\star G,H)\simeq\homs(F,\hhoms_{A}^{\star}(G,H)).$$

We note that we have decorated internal hom objects in $\mathcal{H}_{A/S}$ with a superscript $\star$ to emphasize that the monoidal product is convolution. We will write $\hhoms_{A}$ for the internal hom object with respect to $\otimes.$ If $F$ is an object of $\QC(A),$ we denote its pre-dual object $$\mathbb{D}F:=\hhoms_{A}(F,\mathcal{O}).$$ If $F$ is an object of $\mathcal{H}_{A/S}$, then we denote its pre-dual object $$\mathbb{D}^{\star}F:=\hhoms_{A}^{\star}(F,\mathcal{O}_{e}).$$ We remind the reader that in a closed symmetric monoidal category, if an object $F$ is dualizable then its pre-dual is necessarily also a dual.

\begin{lemma} There are functorial equivalences $$\hhoms_{A}^{\star}(F,G)\simeq \pi_{1,*}\,\hhoms_{A^{2}}(\pi_{2}^{*}\,F,m^{!}G),$$ where $m^!$ denotes a right adjoint\footnote{Note that $m$ is qcqs schematic since $A \to S$ was assumed qcqs schematic.} to $m_*$. \end{lemma}

\begin{proof} We have \begin{align*}
  \homs(H\star F,G) &= \homs(m_{*}\big(\pi_{1}^{*}H\otimes\pi_{2}^{*}F),G\big)\\
  &\simeq \homs(\pi_{1}^{*}H\otimes\pi_{2}^{*}F,m^{!}G)\\
  &\simeq \homs\big(\pi_{1}^{*}H,\hhoms_{A^{2}}(\pi_{2}^{*}F,m^{!}G)\big)\\
  &\simeq \homs(H,\pi_{1,*}\hhoms_{A^{2}}\big(\pi_{2}^{*}F,m^{!}G)\big)
\end{align*} for all $H$. \end{proof}

\begin{corollary}\label{duals} Let $A$ be an abelian group scheme. Let $F$ be an object of $\HA$ which is dualizable as an object of $\QC(A)$. Then $F$ is dualizable in $\HA$, and the dual is given by $$\mathbb{D}^{\star}F\simeq \iota^{*}\,\mathbb{D}F\,[\operatorname{dim}(A)\,].$$\end{corollary}

\begin{proof} First note that $F$ is obtained as an iterated extension of point sheaves, which are invertible with respect to $\star,$ and thus dualizable. It follows then that $F$ is dualizable and we need only compute the pre-dual $\mathbb{D}^{\star}F.$ By the lemma above this is given by \begin{align*} \pi_{1,*}\,\hhoms_{A^{2}}(\pi_{2}^{*}F,m^{!}\mathcal{O}_{e}) &\simeq\pi_{1,*}\big(\pi_{2}^{*}\mathbb{D}F\otimes m^{!}\mathcal{O}_{e}\big)\\
&\simeq \pi_{1,*}\big(\pi_{2}^{*}\mathbb{D}F\otimes\mathcal{O}_{\Gamma_{\iota}}[\,\operatorname{dim}(A)\,]\big),
\end{align*} where $\Gamma_{\iota}$ denotes the graph of the inversion $\iota$. 

We recognize this last expression as the integral transform corresponding to the kernel $\mathcal{O}_{\Gamma_{\iota}}[\,\operatorname{dim}(A)\,]$, applied to $\mathbb{D}F$. It is a general fact that the integral transform with kernel $\mathcal{O}_{\Gamma_{f}}$ equals $f_{*}$; to prove this write $$\mathcal{O}_{\Gamma_{f}}=(\operatorname{id}\times f)_{*}\OO$$ and apply the projection formula. Specializing to our situation, the above expression reduces to $\iota_{*}\mathbb{D}F[\,\operatorname{dim}(A)\,]\simeq \iota^{*}\mathbb{D}F[\,\operatorname{dim}(A)\,]$.

It remains only to be seen that this is still connective. Since connective objects are stable under colimits and extension, this follows from the case of point sheaves. \end{proof}

\begin{remark} For a point $a\in A(k)$, it is not hard to show that the above formula reduces to $$\mathbb{D}^{\star}\mathcal{O}_{a}\simeq\mathcal{O}_{\iota(a)}.$$  \end{remark}

\subsubsection{Classifying stack} If $C$ is any tensor category then there is a stack $\mathbb{B}C$: $$\mathbb{B}C:S\mapsto\operatorname{Fun}^{\otimes}(C,\QC(S))$$ and a tautological tensor functor $$\tau_{C}:C\rightarrow\QC(\mathbb{B}C)$$ classifying the identity morphism of $\mathbb{B}C$. \begin{example}If $\YY$ is Tannakian then we have $\YY\simeq\mathbb{B}\QC(\YY)$ and $\tau_{\QC(\YY)}$ is an equivalence. \end{example}

\begin{lemma}\label{class=dual} Let $A$ be an abelian group scheme. Then there is a canonical isomorphism $$\check{A}\simeq\mathbb{B}\HA.$$\end{lemma}

\begin{proof} We will give a slightly informal account of the proof, as we do not rely explicitly on this result. Consider a tensor functor $$F:\HA\rightarrow\QC(S),$$ representing an $S$-point of $\mathbb{B}\HA$. We must produce an $S$-point of $\check{A}$, which is to say a $S$-family of multiplicative line bundles on $A$. For $a\in A(k)$, the structure sheaf $\OO_{a}$ is invertible in $\HA$, and so $F(\mathcal{O}_{a})=:L_{a}$ is a line bundle on $S$, and as $a$ varies these assemble into a line bundle $L$ on $A\times S$. Monoidality supplies coherent isomorphisms of line bundles on $S$: $$L_{a}\otimes L_{b}\rightarrow L_{a+b}$$ which in turn supply $L$ with a multiplicative structure. By definition this produces an $S$-point of $\check{A}$.

Conversely, let $L$ be an $S$-point of $\check{A}$. Then we produce a tensor functor $\HA\rightarrow\QC(S)$ by taking $$\pi_{S,*}\hhoms_{A\times S}\big(L^{-1},\pi_{A}^{*}(-)\big).$$ One verifies that these constructions are mutually inverse. \end{proof}

\begin{corollary} There is a natural morphism $$\tau_{\HA}:\HA\rightarrow\QC(\mathbb{B}\HA)\simeq\QC(\check{A}).$$ \end{corollary}

\begin{remark} Unlike the case of $\QC(\YY)$, $\tau_{\HA}$ is typically not an equivalence. For example, for $A\in\{\mathbf{G}_a,\GM\}$, recall that $\check{A}$ is given by $B\widehat{\mathbf{G}}_{a}$ and $B\mathbf{Z} \simeq S^1$, respectively. In both cases $\tau_{\HA}$ can be identified with the inclusion \[\Ind(\Perf(\check{A})) \hookrightarrow \QC(\check{A}).\] \end{remark}

\subsubsection{Integrality of $\HA$.}

Let $A$ be an abelian group scheme over a field $K$.

Note that whilst for distinct points $x,y$ of $A$, we have $\mathcal{O}_{x}\otimes\mathcal{O}_{y}=0$, the corresponding statement for $\star$ is false: we have $\OO_{x}\star\OO_{y}=\OO_{x+y}.$ In fact, $\mathcal{H}_{A/K}$ is an integral domain.

\begin{lemma}\label{HE int} The tensor category $\mathcal{H}_{A/K}$ is an integral domain.\end{lemma}
\begin{proof} Recall that non-zero objects of $\mathcal{H}_{A/S}$ are quasi-coherent sheaves whose cohomology sheaves have $0$-dimensional support. Given two such $F$ and $G$, the claim is that $F \star G$ is non-zero. Since the convolution product is left-exact in each argument (by left-exactness of pull-back along the projections $A \times A \to A$ and push-forward along $m : A \times A \to A$), we have $H^i(F \star G) \simeq H^i(F) \star H^i(G)$. We may thus assume that $F$ and $G$ are discrete. Since they have $0$-dimensional support, we may find points $x,y \in A$ and injections $\OO_{x}\rightarrow F$ and $\OO_{y}\rightarrow G$. These give rise to an injection $\OO_{x+y} \rightarrow F \star G$ (using left-exactness of convolution again), so $F \star G$ is nonzero as claimed. \end{proof}

\subsubsection{Global sections functor} Consider the unique symmetric monoidal stable continuous functor $$\mathbf{1}:\DD(k)\rightarrow\DD_{0}(A/S), \quad k \mapsto \OO_e.$$ The right adjoint to $\mathbf{1}$ will be denoted
$$\Gamma_{\SSSA} : \DD_{0}(A) \rightarrow \DD_{k}, \quad F \mapsto \operatorname{Hom}_{\DD(A)}(\OO_{e}, F)$$ and called the functor of \emph{global sections} of $\DD_{0}(A)$. (We recall our heuristic that $\HA$ is the category of quasi-coherent sheaves on a putative stack $\SSSA$.) We will sometimes abuse notation and write $\Gamma_{\SSSA}$ for the induced functor $$\HA\rightarrow\DD_{0}(A)\xrightarrow{\Gamma_{\SSSA}}\DD(k).$$

\begin{remark} Note that $\mathbf{1}$ restricts to a tensor functor $$\mathbf{1}_{+}:\Mod_k \rightarrow\mathcal{H}_{A/S}.$$ On the other hand, $\Gamma_{\SSSA}$ preserves compact objects but not necessarily connective ones. \end{remark}

\subsection{Examples: Rational, trigonometric, and elliptic}

We now specialize to some abelian group schemes $A \to S$ of dimension $1$.  The line bundle $\mathcal{O}(e)$ is written $\vartheta$.

\subsubsection{Rational case}

Take $A = \mathbf{G}_{a,S} \to S$ the additive group. In this case $\mathcal{H}_{A}$ has an action by the dual group $$\check{A}=B\widehat{\mathbf{G}}_{a}.$$ 

\subsubsection{Trigonometric case}

Take $A = \mathbf{G}_{m,S} \to S$ the multiplicative group. In this case $\mathcal{H}_{A}$ has an action by the dual group $$\check{A}=S^{1}.$$

\subsubsection{Elliptic case}

Take $A = E \to S$ a relative elliptic curve. Notice that in this case the dual group $\check{E}$ is identified with $E$. Indeed, it is a standard fact that a line bundle has a multiplicative structure iff it is of degree $0$, and that in this case such a structure is unique and admits no non-trivial automorphisms.

\subsubsection{Fourier--Mukai transform}\label{FMHE} Recall the tautological functor $\tau_{\HE} : \HE \to \QC(\mathbb{B}\HE)$. Under the identification $E\simeq\check{E}\simeq\mathbb{B}\HE$ supplied by Lemma~\ref{class=dual}, this corresponds to a canonical tensor functor \[ \HE \to \QC(E) \] which we now describe more explicitly. Take $S=\spc(k)$ for simplicity.

We denote by $$\FM^{*}:\big(\DD(E),\star\,\big)\rightarrow\big(\DD(E),\otimes\,\big),$$ the Fourier--Mukai transform first introduced in \cite{Mukai}. We adopt the normalization with $$\FM^{*}(\mathcal{O}_{x})=\frac{\mathcal{O}(x)}{\OO(e)},$$ and recall that $\FM^*$ is a symmetric monoidal equivalence.

\begin{lemma} The functor $\FM^*$ restricts to a fully faithful tensor functor $\HE\rightarrow\QC(E)$. \end{lemma}
\begin{proof} We need only check that for $F \in \HE$, $\FM^*(F) \in \DD(E)$ is connective and hence belongs to $\QC(E)$. Since $\HE$ is generated under colimits and extensions by skyscraper sheaves $\OO_{x}$, the claim follows from the fact that $\FM^{*}(\OO_{x})$ lies in the heart of $\DD(E)$. \end{proof}

\begin{remark}To see that $\Phi^{*}:\HE\rightarrow\QC(E)$ agrees with the tautological functor $\tau_{\HE}$, use the fact that the integral kernel for the Fourier--Mukai transform represents the universal family of multiplicative line bundles on $E$.\end{remark}

\section{Elliptic loop spaces}

\subsection{The definition}

We fix an elliptic curve $E$ over $k$.

\begin{definition} The \emph{elliptic loop space} of a tensor category $C$ is the stack $$\hh(C):=\Fun(C,\HE),$$ equipped with the action of $E$ induced from its action on $\HE$. For a stack $\YY$, its elliptic loop space is $$\hh(\YY) := \hh(\QC(\YY)),$$ the elliptic loop space of the tensor category $\QC(\YY)$. \end{definition}

See Subsect.~\ref{ssec:stktenfun} for the definition of the stack of tensor functors, and \ref{sssec:actonHA} for the action of $E\simeq \check{E}$ on $\HE$.

Any tensor functor $f:C\rightarrow D$ induces a morphism of elliptic loop spaces $\hh(f):\hh(D)\rightarrow\hh(C).$ This data naturally assembles into a functor $$\hh:\CCC^{\operatorname{op}}\longrightarrow E\circlearrowright\mathbf{St}_{k},$$ valued in the category of stacks with an action of $E$.

Similarily, if $f:\YY\rightarrow\XX$ is a morphism of stacks, the tensor functor $f^*$ induces a morphism $\hh(f) : \hh(\YY) \to \hh(\XX)$. There is a functor $$\hh:\mathbf{St}_{k}\longrightarrow E\circlearrowright\mathbf{St}_{k}.$$

\begin{remark} We will see below that when $\YY$ is ``nice'', $\hh(\YY)$ is the open of the mapping stack $\Map(E, \YY)$ consisting of maps $f : E \to \YY$ for which $$f^{*}:\QC(\YY)\rightarrow\QC(E)$$ factors through the subcategory $\FM^{*}\HE.$ \end{remark}

\begin{remark} If $E\rightarrow E'$ is a map of elliptic curves then there is an induced tensor map $\mathcal{H}_{E'}\rightarrow\HE$ which induces a functor $L_{E'}(C)\rightarrow\hh(C)$. In particular we can take the multiplication by $d$ map and so produce morphisms $$\operatorname{Adams}^{E}_{d}:\hh(C)\rightarrow\hh(C).$$\end{remark}

\subsection{Loops as maps} For sufficiently nice stacks $\YY$, we now show that the elliptic loop space $\hh(\YY)$ can be realized as an open of the mapping stack $\Map(E,\YY)$.

Given a tensor category $C$, we write $$\mm(C):=\Fun(C,\QC(E))$$ for the stack of tensor functors from $C$ to $\QC(E)$. Composition with the fully faithful tensor functor $\FM^{*}:\HE\rightarrow\QC(E)$ (\ref{FMHE}) induces a natural monomorphism of stacks \[ \hh(C) \to \mm(C). \]
In particular, for \emph{Tannakian} stacks $\YY$ this becomes a natural monomorphism $$\hh(\YY) \to \mm(\YY) \simeq \Map(E,\YY).$$

\begin{theorem}\label{open in map} If $\YY$ is pluperfect then the canonical morphism $\hh(\YY)\rightarrow \Map(E,\YY)$ is an open immersion. \end{theorem}

The notion of pluperfectness will be introduced below. Note that Theorem~\ref{open in map} in particular yields the following geometricity result for $\hh(\YY)$.

\begin{corollary}\label{repr} For any pluperfect Artin stack $\YY$, the elliptic loop space $\hh(\YY)$ is an Artin stack. \end{corollary}
\begin{proof} Since $E$ is a smooth and proper scheme and $\YY$ is Artin, it is well-known that $\Map(E,\YY)$ is Artin. For example, this can be deduced from Lurie's version of Artin's representability theorem (see e.g. the discussion following the proof of Theorem~2.5 in \cite{PTVV}). Hence the claim follows from Theorem~\ref{open in map}. \end{proof}

\begin{remark} By definition, Theorem~\ref{open in map} implies that when $\YY$ is pluperfect the elliptic loop space $\hh(\YY)$ can equivalently be defined as the substack of maps $f : E \to \YY$ for which $$f^{*}:\QC(\YY)\rightarrow\QC(E)$$ factors through the subcategory $\FM^{*}\HE$. \end{remark}

\subsubsection{Pluperfect stacks} Recall that a stack $\YY$ is called \emph{perfect} in \cite{BFN} if it has affine diagonal and is compactly generated by a set of perfect complexes. In particular, perfect stacks are Tannakian by \cite{BHL}. We introduce a tensor variant of this notion:

\begin{definition} We call $\YY$ \emph{pluperfect} if it is Tannakian and there exists a finite set $\{F_{i}\}$ of perfect complexes on $\YY$ which generate the derived category $\DD(\YY)$ under colimits, finite limits, and tensor products. \end{definition}

\begin{example} For $X$ a qcqs algebraic space, $X$ is pluperfect: it is Tannakian by \cite[Thm~9.6.0.1]{SAG} and $\DD(X)$ is compactly generated by a single perfect complex by \cite[Thm.~9.6.3.2]{SAG}. \end{example}

In the case of stacks $\YY$, $\DD(\YY)$ is often perfect but rarely compactly generated by \emph{finitely} many perfect complexes. Nevertheless, we have:

\begin{lemma} Let $X$ be quasi-projective and $G$ an affine algebraic group acting linearly on $X$. Then the quotient stack $X/G$ is pluperfect. \end{lemma}

\begin{proof} Since $G$ is affine, $X/G$ has affine diagonal and is thus Tannakian by \cite{BHL}. Let us show that $\DD(X/G)$ is generated by finitely many perfect complexes. We reduce to $X=\mathbf{P}V$ with $G$-action coming from a morphism $G\rightarrow GL(V)$. Then let $F$ be a $G$-equivariant coherent sheaf on $\mathbf{P}V$ and choose $r$ so that the space of morphisms (of sheaves on $\mathbf{P}V$), $$\OO(-r)\otimes\Gamma(\mathbf{P}V,F(r))\rightarrow F,$$ is non-zero. Let $f$ be a non-zero morphism and observe that it is $G$-equivariant for the $G$ structure on the left coming from the natural such on $\OO(-r)$ and the $G$-action on $\Gamma(\mathbf{P}V,F(r))$, interpreted as a $G$ equivariant structure on the trivial bundle with fibre $\Gamma(\mathbf{P}V,F(r))$. As such we see that $\OO(1)$ and the pull-backs of sheaves on $BG$ generate. Embedding $G$ into some $GL_n$, it remains only to note that the standard representation is a tensor generator of $\DD(GL_n)$ (since we are in characteristic zero) which pulls back to a tensor generator of $BG$. \end{proof}

\subsubsection{Proof of Theorem~\ref{open in map}}

Since $\FM^*$ is fully faithful, $$\hh(\YY)\rightarrow\Map(E,\YY)$$ is a monomorphism with image consisting of those maps $f : E \to \YY$ so that $$f^{*}:\QC(\YY)\rightarrow\QC(E)$$ factors through the subcategory $\FM^{*}\HE.$ We must show that this is an open condition (see e.g. \cite[Def.~19.2.4.1]{SAG}).

Let $S$ be affine and $S\rightarrow\Map(E,\YY)$ classifying a morphism $f : S \times E \to \YY$. For every point $s\in S$ this determines by restriction a map $f_{s} : E_{k(s)} \to \YY_{k(s)}$. The claim is that if $f_{s}^{*}$ factors through $\FM^{*}\HE$ for some point $s$, then there is an open neighbourhood $S'$ of $s$ so that $f_{S'}^{*}$ factors through $$\FM^{*}\HE\otimes\QC(S')\subset\QC(E)\otimes\QC(S') \simeq\QC(E\times S').$$ Equivalently, we may replace $\HE$ and $\QC(-)$ by their stabilizations $\DD_{0}(E)$ and $\DD(-)$, respectively. Since $\YY$ is pluperfect we may choose a finite set $\{F_{i}\}_{i}$ of perfect complexes on $\YY$ which tensor-generate $\DD(\YY)$. By assumption, each $f_{s}^{*}F_{i}$ lies in $\FM^{*}\DD_{0}(E)$. It then suffices to find open neighbourhoods $S'_i$ of $s$ so that each $$f_{S'_{i}}^{*}F_{i} \in \FM^{*}\DD_{0}(E)\otimes\DD(S'_{i}) \subset \DD(E)\otimes \DD(S_{i}')$$ as we may then take $S'$ to be the intersection of the $S'_i$'s. Passing to Fourier--Mukai transforms along the first factor, the claim for each $F_{i}$ reduces to the following lemma.

\begin{lemma}\label{0-open} Let $S$ be an affine scheme, $Y$ a scheme, and $F\in\Perf(Y\times S)$ a perfect complex. If there exists a point $s\in S$ such that the restriction $s^*F = F|_{Y_{k(s)}}$ lies inside $\Perf_{0}(Y_{k(s)})$, then there exists an open neighbourhood $S' \subset S$ of $s$ so that $F|_{S'}$ lies in $\Perf_{0}(Y\times S'/S').$ \end{lemma}

\begin{proof} By definitions, we reduce immediately to the case of $Y$ affine. Replacing $S$ by its local scheme at $s$, it will suffice to show that if $s^*F \in \Perf(Y_{k(s)})$ lies in $\Perf_{0}(Y_{k(s)}/k(s))$, then $F$ lies in $\Perf_{0}(Y \times S/S)$. We may assume $F=F_{Y}\boxtimes F_{S}$ with $F_Y \in \Perf(Y)$ and $F_S \in \Perf(S)$, as $\Perf(Y\times S)$ is generated by such under colimits and extensions. We may assume without loss of generality that $F_{S}$ is non-zero, and so $s^*F_{S} \in \Perf(k(s))$ is non-zero by Nakayama's lemma (see e.g. \cite[Cor.~2.7.4.4]{SAG}). The claim is that the push-forward of $F \in \Perf(Y\times S)$ to $S$, \[ \Gamma(Y,F_Y) \boxtimes F_S \in \QC(\spc(k) \times S) \simeq \QC(S) \] is perfect. We will show that $\Gamma(Y,F_Y) \in \Mod_k$ is perfect.

By assumption, $$F|_{Y_{k(s)}} \simeq F_Y \boxtimes s^*F_S$$ lies in $\Perf_0(Y_{k(s)}/k(s))$, meaning that $$\Gamma(Y,F_Y) \boxtimes s^*F_S$$ is perfect over $k\otimes k(s) \simeq k(s)$. Since $s^*F_S \in \Mod_{k(s)}$ is non-zero, it admits some shift of $k(s)$ as a direct summand. Hence $\Gamma(Y,F_Y) \boxtimes s^*F_S$ admits some shift of $\Gamma(Y,F_Y) \boxtimes {k(s)}$ as a direct summand. It follows that $\Gamma(Y,F_Y) \boxtimes {k(s)}$ is a perfect $k\otimes k(s)$-module, hence by faithfully flat descent $\Gamma(Y,F_Y)$ is a perfect $k$-module. The claim follows. \end{proof}

\subsection{Elliptic loop spaces of schemes} For a scheme $S$ we recall that the unipotent loop space is given by
\[ L^{\operatorname{uni}}(S) := \Map(B\mathbf{G}_a, S). \]
As explained in \cite{BNloop}, this is canonically identified with the $(-1)$-shifted tangent bundle
\[ \TT[-1]S := \spc_S(L\Omega_S), \]
the relative spectrum of the derived de Rham complex $L\Omega_S$. More precisely, we have $L^{\operatorname{uni}}(S) \simeq \spc_S(\operatorname{Sym}^*_{\OO_S}(\mathbf{L}_S[1]))$ as schemes, and the action of the group stack $B\mathbf{G}_a$ on $L^{\mathrm{uni}}(S)$ corresponds to the de Rham differential.

\begin{theorem}\label{schemes} Let $S$ be a qcqs scheme. Then there are isomorphisms $$\hh(S)\simeq L^{\mathrm{uni}}(S) \simeq \TT[-1]S.$$ Moreover, the $E$-action on $\hh(S)$ factors canonically through the affinization $\operatorname{Aff}(E)\simeq B\mathbf{G}_{a}$, and the above upgrades to an isomorphism of $B\mathbf{G}_a$-equivariant schemes. \end{theorem}

\begin{proof} Using Corollary~\ref{Zarcosh} we reduce to the case of $S=\spc(R)$ affine. Note that a tensor functor $$\Mod_{R}\rightarrow\HE\otimes \Mod_{B}$$ is specified entirely by a map of rings $$R\rightarrow Z(\HE\otimes\Mod_{B})\simeq Z(\HE)\otimes B$$ since $R\in\Mod_R$ is a generator with endomorphism algebra $R$, which by symmetric monoidality must be sent to the unit object in $\HE\otimes\Mod_{B}$. The monoidal centre $Z(\HE)$, by definition the endomorphism algebra of the tensor unit $\mathcal{O}_{e}$, is given by the algebra $k[\epsilon]$ with $\epsilon$ of degree $+1$. In particular we find that $$\hh(S) \simeq \Map_{\operatorname{CAlg}_{k}}(R,k[\epsilon]) \simeq \Map(B\mathbf{G}_a, S),$$ since $B\mathbf{G}_a \simeq \mathbf{G}_a[1] \simeq \spc(k[\epsilon])$. This is $B\mathbf{G}_a$-equivariantly isomorphic to $\TT[-1]S$, see e.g. \cite[\S 4]{BNloop}. Finally, since $\TT[-1]S$ is affine, so is $\hh(S)$, so that the action of $E$ on $\hh(S)$ factors through the affinization $\operatorname{Aff}(E)\simeq B\mathbf{G}_{a}$.
\end{proof}

\subsection{Mayer--Vietoris and K\"unneth} We establish some fundamental properties of the assignment $\YY \mapsto \hh(\YY)$.

\begin{lemma}\label{open} If $\YY$ is Tannakian and $\mathcal{U}\rightarrow\YY$ is a quasi-compact open immersion with $\UU$ Tannakian, then $\hh(\mathcal{U})\rightarrow\hh(\YY)$ is a quasi-compact open immersion. \end{lemma}

\begin{proof} Consider the commutative square \[\begin{tikzcd}
\hh(\mathcal{U}) \arrow{r} \arrow[swap]{d}& \mm(\mathcal{U}) \arrow{d}\\
\hh(\YY)\arrow{r} & \mm(\YY).
\end{tikzcd}\]
Since $\QC(\YY)\rightarrow\QC(\mathcal{U})$ is a tensor localization and $\HE\rightarrow\QC(E)$ is fully faithful, Lemma~\ref{Cart} implies that this square is cartesian. Recall that when $\YY$ is Tannakian, we may identify $\Map(E,\YY)\simeq\mm(\YY)$. It is well-known that $\Map(E,\mathcal{U}) \to \Map(E,\YY)$ has the asserted property (see e.g. the proof of Proposition~19.1.1.1 in \cite{SAG}), so the claim follows by base change. \end{proof}

\begin{theorem}[Mayer--Vietoris]\label{MVK} If $\YY$ is Tannakian and $\{\UU_{i}\}$ is a finite cover by Tannakian quasi-compact opens, then $\{\hh(\UU_{i})\}$ is a cover of $\hh(\YY)$ by quasi-compact opens. \end{theorem}

\begin{proof} By Lemma~\ref{open}, each of the maps $$\hh\,\mathcal{U}_{i}\rightarrow\hh\,\YY$$ is open. We now show that the coproduct map $$\coprod_{i}\hh\,\mathcal{U}_{i}\rightarrow\hh\,\YY$$ is surjective on $K$-points for every field extension $K/k$. In particular, it is an effective epimorphism of étale sheaves.

Given a $K$-point of $\hh(\YY)$ classifying a tensor functor $$\phi : \QC(\YY)\rightarrow\HE\otimes K \simeq \mathcal{H}_{E_K/K},$$ we must show that $\phi$ factors through some $\QC(\mathcal{U}_{i})$. For this we note that $$\{\QC(\mathcal{U}_{i})\}_{i}$$ forms a cover of $\QC(\YY)$ in the sense of definition \ref{cover}, as the restrictions $\QC(\YY) \to \QC(\UU_i)$ are tensor localizations (since $\UU_i$ are quasi-compact opens) and they are jointly conservative (by faithfully flat descent). Since $\mathcal{H}_{E_K/K}$ is an integral domain (Corollary~\ref{HE int}), we may conclude by appealing to Lemma~\ref{local}. \end{proof}

\begin{theorem}[K\"unneth]\label{Kun} The functor $\hh(-)$ sends every cartesian square of stacks \[\begin{tikzcd}
  \YY_1 \times_{\XX} \YY_2 \ar{r}\ar{d} & \YY_2 \ar{d} \\
  \YY_1 \ar{r} & \XX,
\end{tikzcd}\] which are perfect in the sense of \cite{BFN}, to a cartesian square in $E\circlearrowright\mathbf{St}_{k}$. \end{theorem}

\begin{proof} Under the perfectness assumption we have \[\QC(\YY_1 \times_{\XX} \YY_2) \simeq \QC(\YY_1) \otimes_{\QC(\XX)} \QC(\YY_2), \] passing to connective subcategories from \cite[Thm.~1.2(i)]{BFN}. Thus the claim follows from Lemma~\ref{fibre}. \end{proof}

\begin{corollary}\label{Zarcosh} Let $\YY$ be a quasi-compact pluperfect stack. Then the assignment $\UU \mapsto \hh(\UU)$ determines a cosheaf for the Zariski topology on $\YY$. \end{corollary}
\begin{proof} Combine Theorems~\ref{MVK} and \ref{Kun}. \end{proof}

\begin{corollary}\label{mapsrepr} Let $f : \XX \to \YY$ be a morphism of pluperfect stacks. If $f$ is schematic, then so too is $L_{E}(f)$. \end{corollary}
\begin{proof} Since $\hh(\XX)$ and $\hh(\YY)$ are Artin by Corollary~\ref{repr}, the morphism $\hh(f)$ is representable by Artin stacks. Thus it is schematic if and only if for every \emph{classical} scheme $S$ and every morphism $S \to \hh(\YY)$, the classical truncation of $\hh(\XX)\times_{\hh(\YY)} S$ is a scheme. Since $(L_{E}S)^\cl \simeq S$ by Theorem~\ref{schemes}, Theorem~\ref{Kun} yields \[ \hh(\XX\times_{\YY}S)^{\mathrm{cl}} \simeq (\hh(\XX)\times_{\hh(\YY)} \hh(S))^{\mathrm{cl}} \simeq (\hh(\XX)\times_{\hh(\YY)} S)^{\mathrm{cl}}. \] Since $\XX\times_\YY S$ is a scheme by assumption, Theorem~\ref{schemes} implies in particular that the left-hand is also a scheme. The claim follows. \end{proof}

\begin{lemma}\label{proper} Let $f : \XX \to \YY$ be a morphism of pluperfect stacks. If $f$ is proper and schematic, then so too is $L_{E}(f)$. \end{lemma}
\begin{proof} By Theorem~\ref{Kun} and Corollary~\ref{mapsrepr} it suffices to show that for any proper scheme $S$, the scheme $L_{E}(S)$ is proper. The result then follows from the natural isomorphism $L_{E}S\simeq\TT[-1]S$ of Theorem~\ref{schemes}, which shows in particular that the classical truncation of $L_{E}S$ is $S$. (Recall that properness is detected on classical truncations.) \end{proof}

\subsection{Tangent complexes} We turn our attention to the infinitesimal theory, and compute the tangent complex of the elliptic loop space $\hh(\YY)$.
To motivate our description, we begin with an informal computation in terms of the ansatz \[\hh(\YY) \simeq \Map(\SSS,\YY).\] Considering the correspondence
\[ \begin{tikzcd}
    & \hh(\YY)\times\SSS \ar{rd}{\eval}\ar[swap]{ld}{\mathrm{pr}} & \\
    \hh(\YY) & & \YY,
\end{tikzcd} \]
where we imagine that the hypothetical space $\SSS$ is smooth proper, the standard formula for the tangent complex of a mapping stack yields $$\TT_{\hh(\YY)} \simeq \mathrm{pr}_* \eval^{*}\TT_{\YY}.$$
Recall that the key point is that a first-order deformation of $f:\SSS\rightarrow\YY$ is given by a compatible choice of tangent vectors along the image of $f$, which amounts to an element of $\Gamma(\SSS ,f^{*}\TT_{\YY})$.

\subsubsection{The functor $\Theta$} We begin by introducing an analogue of the functor $\mathrm{pr}_* \eval^*$ above.

Recall from \ref{sssec:coev} that for any tensor category $C$ there is a \emph{coevaluation} functor $$\phi_{C}:C\rightarrow\HE\otimes\QC(\hh(C)).$$ Composing with global sections, this induces a functor \begin{align*} \Theta_C : C \xrightarrow{\phi_C} \HE\otimes\QC(\hh(C)) \xrightarrow{\Gamma_{\SSS}} &\DD(k) \otimes\QC(\hh(C))\\ \simeq &\DD(\hh(C)). \end{align*} If $C=\QC(\YY)$ we write $\Theta_{\YY}:=\Theta_{C}.$ We omit the subscripts when there is no risk of ambiguity.

\begin{remark} Note that $\Theta_C$ can in fact be refined to a functor valued in the tensor category $\DD^E(\hh(C)) := \DD(\hh(C)/E)$ of \emph{$E$-equivariant} quasi-coherent sheaves. \end{remark}

\begin{remark} For a tensor functor $f:C\rightarrow D$, there is a commutative square \[\begin{tikzcd}
  C \ar{r}{\Theta_C}\ar{d}{f} & \DD^E(\hh(C)) \ar{d}{\hh(f)^*}\\
  D \ar{r}{\Theta_D} & \DD^E(\hh(D)).
\end{tikzcd}\] \end{remark}

\begin{remark} For a stack $\YY$, $\Theta_\YY$ extends to a functor $$\Theta_{\YY}:\DD(\YY)\rightarrow\DD^{E}(\hh(\YY))$$
on stabilizations, which we denote the same way. \end{remark}

\subsubsection{Tangent complexes.}

\begin{theorem}\label{tangent} Let $\YY$ be pluperfect and assume it admits a tangent complex $\TT_{\YY}$. Then $\hh(\YY)$ admits a tangent complex, and moreover we have $$\TT_{\hh(\YY)}=\Theta_{\YY}(\TT_{\YY}).$$ In particular, if $\TT_{\YY}$ is of amplitude $[a,b]$, then $\hh(\YY)$ is of amplitude $[a-1,b]$. \end{theorem}

\begin{proof} By Theorem~\ref{open in map}, $\TT_{\hh(\YY)}$ is the restriction of the tangent complex of the mapping stack $\Map(E,\YY)$. For the latter the analogous statement is well-known, see e.g. \cite[\S 2.1]{PTVV}). \end{proof}

\begin{corollary}\label{reltangent} Let $f:\XX\rightarrow\YY$ be a morphism of pluperfect stacks. If $\XX$ and $\YY$ have perfect tangent complexes, then $\hh(f):\hh(\XX)\rightarrow\hh(\YY)$ has perfect relative tangent complex given by \[ \TT_{\hh(\XX)/\hh(\YY)} \simeq \Theta_\XX(\TT_{\XX/\YY}). \] In particular, if $f$ is smooth then $\hh(f)$ is quasi-smooth. \end{corollary}
\begin{proof} By above, the relative tangent is given by the fibre of \[ \Theta_\XX(f^*\TT_{\XX}) \simeq \hh(f)^*(\Theta_\YY(\TT_\YY)) \to \Theta_\XX(\TT_\XX), \] which by exactness of $\Theta_\XX$ is $\Theta_\XX(\TT_{\XX/\YY})$. \end{proof}

\subsection{Push-forwards}

In this section we discuss some basic functorialities of the theory $\hh$. We define push-forwards, which depend on a twist by a line bundle. 

\subsubsection{Trace maps}

Let $f:\XX\rightarrow\YY$ be an eventually coconnective morphism of Artin stacks. We write $\omega_{f}=f^{!}\mathcal{O}_{\YY}$ for the relative dualizing complex. When $f$ is proper, we have the adjunction counit $$f_{*}\omega_{f}\rightarrow\mathcal{O}_{\YY}.$$
Recall that if $f$ is quasi-smooth, or $\XX$ and $\YY$ are quasi-smooth, the dualizing complex is given by $\omega_{f} \simeq \operatorname{det}\mathbf{L}_{f},$ where $\operatorname{det}$ denotes the graded determinant, see \cite[App.~B]{HalD}. In that case the counit defines the \emph{trace} map $$\operatorname{tr}_{f} : f_*\det(\mathbf{L}_f) \to \OO_{\YY}.$$

\subsubsection{Elliptic push-forwards}

Let $\XX$ and $\YY$ be pluperfect stacks and $f : \XX \to \YY$ a morphism.

By Lemma~\ref{proper}, if $f$ is proper then so is the induced map $\hh(f) : \hh(\XX)\to\hh(\YY)$.

When $f$ is smooth, or more generally $\XX$ and $\YY$ are smooth over a common base, we get $$\omega_{\hh(f)}\simeq \det(\mathbf{L}_{\hh(f)}) \simeq \vartheta(\mathbf{L}_{f})$$ by dualizing Corollary~\ref{reltangent} and taking determinants, where we have denoted $\vartheta(-) := \operatorname{det}\Theta(-)$.

Thus, when both of the above assumptions hold, we have a trace map for $\hh(f)$, which we denote $$ \operatorname{tr}^{E}_{f} := \operatorname{tr}_{\hh(f)} : \hh(f)_*\vartheta(\mathbf{L}_{f}) \to \OO_{\hh(\YY)}. $$

\begin{definition}\label{push} The \emph{elliptic push-forward} along $f:\XX\rightarrow\YY$ is the above map $\operatorname{tr}^{E}_{f} : \hh(f)_*\vartheta(\mathbf{L}_{f}) \to \OO_{\hh(\YY)}$ of quasi-coherent sheaves on $\hh(\YY)$. \end{definition}

\begin{remark} By construction, $\etr_{f}$ is compatible with $E$-actions, as all involved objects are naturally $E$-equivariant. \end{remark}

\begin{remark}The following properties of $\etr_{f}$ are inherited from the corresponding properties of trace maps.
\begin{itemize}
\item The assignment $f\mapsto\etr_{f}$ is functorial, modulo the canonical identifications $$\vartheta(\mathbf{L}_{gf})\simeq\vartheta(\mathbf{L}_{f})\otimes\hh(g)^{*}\vartheta(\mathbf{L}_{g})$$ given composable morphisms $f$ and $g$.
\item $\etr_{f}$ satisfies a projection formula with respect to pull-backs.
\end{itemize} \end{remark}

\subsection{Rational and trigonometric loop spaces} The construction $\hh(-)$ can be generalized as follows for any abelian group scheme $A$. Let $C$ be a tensor category and $\YY$ a stack. Then we define $$L_{A}C:=\Fun(C,\HA),\quad L_{A}\YY:= L_{A}\QC(\YY).$$ 
 
For the additive and multiplicative group schemes, we have:

\begin{lemma} If $\YY$ is perfect, then we have $$L_{\mathbf{G}_{a}}\YY\simeq L^{\operatorname{rat}}\YY, \quad L_{\GM}\YY\simeq L\YY.$$\end{lemma}

\begin{proof} Let us take $A=\GM.$ Since $\YY$ is Tannakian, we have \[L\YY := \Map(S^{1},\YY) \simeq\, \Fun\big(\QC(\YY),\QC(S^{1})\big).\]

Since $\YY$ is perfect, any tensor functor $\QC(\YY)\to\QC(S^{1})$ factors uniquely through $\Ind\Perf(S^{1})$, as it must preserve dualizable objects.

In particular, we have
\[\Fun\big(\!\QC(\YY),\QC(S^{1})\big) \simeq \Fun\big(\!\QC(\YY),\mathcal{H}_{\GM}\big) = L_{\GM}(\YY). \]
The proof in the case of $A=\mathbf{G}_{a}$ is essentially identical, using $\mathcal{H}_{\mathbf{G}_a} \simeq \Ind\Perf(B\widehat{\mathbf{G}}_{a})$. \end{proof}

\section{Equivariant elliptic Hodge cohomology} We summarize the applications to global quotient stacks $X/G$ with $X$ quasi-projective and $G$ an affine algebraic group acting linearly. We denote the category of such $\mathbf{Var}^G_{k}$. The equivariant geometry of $G$ acting on $X$ is encoded in the morphism $$\pi_{X}^{G}:X/G\rightarrow BG,$$ where we note that $X$ is recovered, together with its $G$-action, as the fibre over $*\rightarrow BG$.

\subsection{Definition and properties}

\begin{definition} The \emph{$G$-equivariant elliptic loop space of $X$} is defined as $$\hhg(X):=\hh(X/G).$$ We regard this a stack with $E$-action, equipped with an $E$-equivariant structure morphism $$\hh(\pi^{G}_{X}):\hhg(X)\longrightarrow L_{E}^{G}(*).$$\end{definition} 

\begin{definition} The \emph{$G$-equivariant elliptic Hodge cohomology of $X$} is defined as the functor $$HH^{G}_{E}: (\mathbf{Var}^G_{k})^\mathrm{op} \rightarrow \QC^{E}(\hhg(*))$$ given by $$HH^{G}_{E}(X):=\hh(\pi^{G}_{E})_{*}\OO.$$ \end{definition}

\begin{theorem} $HH^{G}_{E}$ satisfies Mayer--Vietoris and Kunneth. There are natural restriction functors for maps $G\rightarrow G'$ and change of curve functors for maps $E'\rightarrow E$.\end{theorem}

\begin{proof} The first part is immediate from Theorems~\ref{MVK} and \ref{Kun}. The second is a direct consequence of the functoriality of $\YY\mapsto L_{E}\YY$ with respect to both $\YY$ and $E$. \end{proof} 

\begin{remark} The cohomology theory $HH^{G}_{E}(-)$ fails to satisfy homotopy invariance with respect to the affine line $\mathbf{A}^1$. Indeed this fails even for the trivial group $G=\{e\}$: the elliptic Hodge cohomology of $\mathbf{A}^1$ is by Theorem~\ref{schemes} the Hodge cohomology of $\mathbf{A}^{1}$, which does not agree with Hodge cohomology of a point. \end{remark}

\subsection{Example: $G=GL_r$}

The category $\QC^E(L^G_E(*))$ can be identified more explicitly. We focus on the case $G=GL_{r}$ for simplicity. Write $\HE^{\omega,\heartsuit} \subset \HE^{\heartsuit}$ for the abelian categories of coherent and quasi-coherent sheaves on $E$ with $0$-dimensional support, respectively.

\begin{lemma} \label{GL} $\hh^{GL_{r}}(*)$ is the classical Artin stack classifying length $r$ objects of $\HE^{\omega,\heartsuit}$. \end{lemma}

\begin{proof} By Theorem~\ref{open in map}, $\hh^{GL_r}(*)$ is open in the mapping stack $$\Map(E,BGL_{r}).$$ Since $E$ is a smooth curve, the latter is a smooth classical stack. In particular, $\hh^{GL_r}(*)$ is isomorphic to its classical truncation, which by Lemma~\ref{abelian} is the stack of tensor functors $\operatorname{Rep}^{\heartsuit}(GL_{r}) \to \HE^{\heartsuit}$. Basic representation theory of $GL_{r}$ implies that any such functor is determined by the image of the standard representation, which is a dualizable object with monoidal trace $r$ and so must map to such.

We claim that dualizable objects in $\HE$ of trace $r$ are precisely the length $r$ coherent sheaves on $E$. Indeed any coherent sheaf is dualizable in $\QC(E)$, as $E$ is smooth, and Corollary~\ref{duals} shows that $F$ is also dualizable in $\HE$. To compute the monoidal trace of $F$ we reduce to the case of length one sheaves and use additivity for exact triangles, recalling that any dualizable object of $\HE^{\heartsuit}$ is an iterated extension of length one objects. For length one objects the claim is clear: any such is $\otimes$-invertible and so the evaluation and coevaluation are both the identity map, whence the trace is $1$ as desired. \end{proof}

\begin{remark} In fact a striking theorem of Iwanari, \cite{Iwa}, computes the stack of stable symmetric monoidal functors from $\operatorname{Rep}(GL_{r})$ to a general $C$ as the stack parameterizing objects $x$ with $\wedge^{r}x$ invertible and $\wedge^{r+1}x\simeq 0$. Inserting connectivity conditions this recovers the above as a (very) special case. \end{remark}

\begin{remark} Under the Fourier--Mukai transform $\FM^*$, the stack of length $r$ coherent sheaves on $E$ with $0$-dimensional support admits an alternative description originally due to Atiyah (see \cite{At}): it is isomorphic to the stack $$\operatorname{Bun}^{\operatorname{ss}}_{r,0}(E)$$ of semistable rank $r$ degree $0$ vector bundles on $E$. Moreover $G=GL_{r}$ can be replaced by a general reductive group, where we recall that semistability for a $G$-bundle is defined in terms of the adjoint $GL(\mathfrak{g})$-bundle. \end{remark}

\begin{remark} Taking supports of length $r$ coherent sheaves determines a canonical morphism to the symmetric power of $E$, $$\mathrm{supp} : \hh^{GL_{r}}(*)\rightarrow E^{(r)},$$ which is a good moduli space. Note that this morphism is $E$-equivariant for the trivial action on the target. As another example, we note that there is a good moduli space $$\mathrm{supp}:L_{E}^{SL_{r}}(*)\rightarrow E^{(r)}_{0},$$ where $E^{(r)}_{0}$ subset of $E^{(r)}$ consisting of tuples of sum $0$.\end{remark}

\label{springer}\subsection{Elliptic Springer resolution}

Consider the universal flag bundle $$BB\rightarrow BGL_{r}.$$ The elliptic loop space $L_{E}^{B}(*) = \hh(*/B)$ is the classical stack of tensor functors $$\operatorname{Rep}(B)\rightarrow\HE,$$ or equivalently $r$-step flags of one-dimensional objects of $\HE$. Under the Fourier--Mukai transform $\FM^*$ this is identified with the connected component $$\operatorname{Bun}^{0}_{B}(E)\subset\operatorname{Bun}_{B}(E)$$ which parameterizes flag bundles all of whose associated graded line bundles have degree $0$. Thus the map \[ L_E^B(*) \to L_E^G(*) \] is precisely the \emph{elliptic Grothendieck--Springer resolution} of \cite{BNEll}.

\end{document}